\documentclass{amsart}

\usepackage{amsfonts,amsmath,amsthm,amssymb,amscd,latexsym,enumerate,etoolbox,mathrsfs,comment}

\usepackage[all,cmtip,pdf]{xy}
\usepackage{thmtools, thm-restate}
\usepackage{tikz-cd}
\usepackage{tikz}
\usetikzlibrary{calc}

\usetikzlibrary{shapes,arrows,decorations.markings}

\tikzset{->-/.style={decoration={
  markings,
  mark=at position .525 with {\arrow{Straight Barb}}},postaction={decorate}}}
\tikzset{->>-/.style={decoration={
  markings,
  mark=at position 0.5 with {\arrow{Straight Barb}},
  mark=at position 0.6 with {\arrow{Straight Barb}}},postaction={decorate}}}

\newcommand{\R}{\mathbb{R}}
\newcommand{\C}{\mathbb{C}} 
\newcommand{\N}{\mathbb{N}}

\newcommand{\Z}{{\mathbb Z}}

\renewcommand{\S}{\mathscr{S}}

\newcommand{\I}{\mathcal{I}}

\renewcommand{\phi}{\varphi}

\newcommand{\Per}{\text{Per}}

\newcommand{\diam}{{\,\text{diam}}}
\newcommand{\interior}{\text{int}}
\newcommand{\lang}{\mathcal{L}}
\newcommand{\ev}{{\text{even}}}

\newcommand{\parens}[1]{\left( #1 \right)}

\newcommand{\word}[1]{\mathtt{#1}}
\newcommand{\ie}{\textit{i.e.}\ }

\newcommand{\lc}{\text{lc}}
\newcommand{\lcu}{\text{lcu}}
\newcommand{\lcs}{\text{lcs}}
\newcommand{\sync}{\text{sync}}

\theoremstyle{plain}
    \newtheorem{theorem}{Theorem}[section]
    \newtheorem{lemma}[theorem]{Lemma}
    \newtheorem{corollary}[theorem]{Corollary}
    \newtheorem{proposition}[theorem]{Proposition}
    \newtheorem{prop}[theorem]{Proposition}
    
\theoremstyle{definition}
    \newtheorem{definition}[theorem]{Definition}
    \newtheorem{example}[theorem]{Example}

    \newtheorem{remark}[theorem]{Remark}
    
\theoremstyle{remark}

\begin{document}

\title[Synchronizing Dynamical Systems]{Synchronizing Dynamical Systems: their groupoids and $C^*$-algebras}
\author{Robin J. Deeley}
\address{Robin J. Deeley,   Department of Mathematics,
University of Colorado Boulder
Campus Box 395,
Boulder, CO 80309-0395, USA }
\email{robin.deeley@colorado.edu}
\author{Andrew M. Stocker}
\address{Andrew M. Stocker,   Department of Mathematics,
University of Colorado Boulder
Campus Box 395,
Boulder, CO 80309-0395, USA }
\email{andrew.stocker@colorado.edu}
\thanks{Both RJD and AMS were partially supported by NSF Grant DMS 2000057.}
\begin{abstract}
Building on work of Ruelle and Putnam in the Smale space case, Thomsen defined the homoclinic and heteroclinic $C^\ast$-algebras for an expansive dynamical system.  In this paper we define a class of expansive dynamical systems, called synchronizing dynamical systems, that exhibit hyperbolic behavior almost everywhere. Synchronizing dynamical systems generalize Smale spaces (and even finitely presented systems). Yet they still have desirable dynamical properties such as having a dense set of periodic points. We study various $C^\ast$-algebras associated with a synchronizing dynamical system. Among other results, we show that the homoclinic algebra of a synchronizing system contains an ideal which behaves like the homoclinic algebra of a Smale space. 
\end{abstract}

\maketitle

\section*{Introduction}
The theories of $C^\ast$-algebras and dynamical systems are connected at a deep and fundamental level. $C^\ast$-algebras can be constructed from dynamical systems by associating a groupoid to the dynamical system and then forming the associated $C^\ast$-algebra(s). This allows for the construction of many interesting $C^\ast$-algebras. In addition, the $K$-theory of the associated $C^\ast$-algebra is an invariant of the original dynamical system and hence can be used to distinguish different dynamical systems.

One important example of this process is the use of $C^\ast$-algebras in the study of Smale spaces, an important class of dynamical systems. In the present paper, a dynamical system is a pair $(X, \varphi)$ where $X$ is a compact metric space and $\varphi: X \rightarrow X$ is a homeomorphism. A Smale space is a dynamical system that is hyperbolic in a uniform way; this is made precise using a map called the bracket. Ruelle \cite{Ruelle1988} and Putnam \cite{putnam_1996} introduced $C^\ast$-algebras associated to a (mixing) Smale space and based on the efforts of many people these $C^\ast$-algebras are at this point well-understood at least from the perspective of Elliott's classification program, see \cite{putnam99, MR4069199, MR3766855}. 

Based on the success in the Smale space case, it is natural to consider generalizations to larger classes of dynamical systems. There are (at least) two natural choices. Firstly, one can consider finitely presented systems \cite{fried1987}, which are quite close to Smale spaces in terms of their dynamical properties. Secondly, one could consider expansive dynamical systems. A dynamical system, $(X, \varphi)$, is expansive if there is a constant $\varepsilon_X > 0$, such that for any $x,y \in X$, \[ d(\varphi^n(x),\varphi^n(y)) \leq \varepsilon_X \text{ for all } n \in \mathbb{Z} \] implies $x = y$.

This definition is deceptively simple. On the one hand, the class of expansive dynamical systems is incredibly large and they are in particular a vast generalization of Smale spaces. On the other hand, expansive dynamical systems retain some form of hyperbolic structure. A precise example of this last informal statement is the existence of the adapted metric on an expansive dynamical system, see \cite{fried1987}. To summarize the relationship, we have that
\[ (X,\varphi) \text{ is a Smale space } \implies (X,\varphi) \text{ is finitely presented } \implies (X,\varphi) \text{ is expansive } \] 

In \cite{thomsen2010c}, Thomsen generalizes the constructions of Ruelle and Putnam to any expansive dynamical system. However at present the $C^\ast$-algebras associated with an expansive dynamical system are not well-understood. Based on examples it is clear that many techniques used to study Smale spaces and their $C^\ast$-algebras do not generalize to the collection of all expansive dynamical systems.  For example Toeplitz flows are minimal shift spaces --- hence they are expansive systems which, in the non-trivial case, do not contain periodic points \cite{downarowicz05}.  This makes Thomsen's construction of the heteroclinic algebras not applicable, see \cite[Chapter 4]{thomsen2010c}.

The goal of this paper is two-fold. Firstly, we isolate a class of expansive dynamical systems that is more general than Smale spaces (and even more general than finitely presented systems), but for which many Smale space techniques can still be applied. These systems will be called synchronizing dynamical systems. Our second goal is to study the structure of Thomsen's $C^\ast$-algebras for this class and contrast these results with the Smale space case.

A point in an expansive dynamical system is called synchronizing if it has a local product neighborhood, see Section \ref{SecSynDS} for the precise definition. The importance of local product structure is well established in the study of dynamical systems. However, the formulation given here (which very much builds on work of Fried \cite{fried1987}) seems to be new; the term synchronizing is rooted in the theory of subshifts and the precise connection is discussed in \cite{AndrewPhDThesis}. 

A synchronizing dynamical system (which we refer to as a synchronizing system) is an irreducible expansive dynamical system that has at least one synchronizing point. It follows from the definitions of irreducible and synchronizing point that the set of synchronizing points in a synchronizing system is open and dense. Furthermore, Smale spaces are exactly systems in which each point is synchronizing. Also, every irreducible finitely presented system is synchronizing. However, even when one restricts to the study of subshifts, there are many interesting examples of synchronizing systems that are not finitely presented. In summary, synchronizing systems are a significant generalization of Smale spaces, but the existence of a local product neighborhood for ``most" of the points allows one to apply Smale space type techniques.

With the relevant class of dynamical systems now determined, our results take three general forms: purely dynamical, groupoid related, and $C^\ast$-algebraic. The most important dynamical result is that the set of periodic points is dense in a synchronizing system. This is also the most involved proof in the paper. This result is related to a result of Artigue, Brum, and Potrie in \cite{artigue2008local}. There are two differences between them. Firstly, the result in \cite{artigue2008local} only applies to manifolds. Secondly, the result in \cite{artigue2008local} assumes that there is a dense set of hyperbolic periodic points and proves that there is an open and dense set of points with local product structure. Our result is in the other direction and does not make any assumption about the underlying space (beyond it being a compact metric space).

The groupoid and $C^\ast$-algebraic results are very much intertwined. A prototypical result is that the set of synchronizing points is invariant under the local conjugacy relation; the $C^\ast$-algebraic implication is that there is an ideal corresponding to the restriction of the local conjugacy relation to the synchronizing points. This ideal behaves in many ways like the homoclinic algebra of a Smale space.

As mentioned above, (irreducible) finitely presented systems are synchronizing. For such systems, our results take their most definitive form (the reader will find all relevant definitions in the main body of the paper):

\begin{theorem}
Suppose $(X,\varphi)$ is a mixing finitely presented system and $P \subseteq X$ a finite set of synchronizing periodic points. Then, the groupoids $G^{lc}_{\sync}(X, \varphi)$, $G^\lcs(X,\varphi,P)$, and $G^\lcu(X,\varphi,P)$ are each amenable and \[ 0 \longrightarrow \I_\sync(X,\varphi) \longrightarrow A(X,\varphi) \longrightarrow A(X,\varphi)/\I_\sync(X, \varphi) \longrightarrow 0 \] is an exact sequence of $C^\ast$-algebras, where
\begin{enumerate}[(i)]
    \item $\I_\sync(X,\varphi)$ is Morita equivalent to $S(X,\varphi,P) \otimes U(X,\varphi,P)$ and
    \item $\I_\sync(X,\varphi)$, $S(X,\varphi,P)$, and $U(X,\varphi,P)$ are all simple.
\end{enumerate}
\end{theorem}

In contrast, for a mixing Smale space, the homoclinic algebra is simple and is Morita equivalent to the tensor of product of the stable and unstable algebras (that is, $A(X,\varphi)$ is Morita equivalent to $S(X,\varphi,p) \otimes U(X,\varphi,p)$ in the Smale space case). Thus, it is the ideal $\I_\sync(X, \varphi)$ that has similar properties to the homoclinic algebra of a Smale space. On the other hand, there are examples where $\I_\sync(X, \varphi)$ is not Morita equivalent to the homoclinic algebra of a Smale space. This means that we are dealing with a strictly larger class of $C^\ast$-algebras. This particular result along with many more about synchronizing shift spaces can be found in \cite{shift-paper} (also see \cite{AndrewPhDThesis}). 

There are many open questions concerning synchronizing systems (and more generally expansive systems) and the associated groupoids/$C^\ast$-algebras. For example, an important open problem is whether the homoclinic groupoid of a general expansive system is amenable. This question is even open in the finitely presented case; although we do prove that the groupoid associated to the synchronizing points is amenable in the finitely presented case and prove that the homoclinic groupoid is amenable when there are finitely many non-synchronizing points and the system is finitely presented.

The structure of the paper is as follows. We begin with the required preliminaries. This includes a discussion of expansive dynamical systems and the subclasses of Smale spaces and finitely presented systems. The precise formulation of a local product structure is developed in Section \ref{sec:local-stable-unstable}. Much of this is based on work of Fried \cite{fried1987}. The definition of synchronizing dynamical system and a proof that periodic points are dense in such a system is given in Section \ref{SecSynDS}. Various groupoids and their $C^\ast$-algebras are introduced and studied in Sections \ref{SecGroupoid} and \ref{c-star-algebras}. This builds on work of a number of people, most notably Putnam, Ruelle, and Thomsen. Finally, the structure of the homoclinic algebra of a mixing finitely presented system is studied in Section \ref{SecFPS}. 


\section*{Acknowledgments}

The authors thank Ian Putnam for interesting and insightful discussions. We also thank the referee for their careful reading of the paper and useful suggestions to improve it.

\section{Preliminaries}

\subsection{Expansive dynamical systems}
In our context, a dynamical system is a pair $(X, \varphi)$ where $X$ is a compact metric space and $\varphi: X \rightarrow X$ is a homeomorphism. We will typically assume that $X$ is infinite. If $x\in X$, then the orbit of $x$ is the set 
\[ \{ \varphi^n(x) \mid n \in \mathbb{Z} \} \,. \]
\begin{definition}\label{def:factor}
A \emph{factor map} between the dynamical systems $(Y,\psi)$ and $(X,\varphi)$ is a surjective continuous map $\pi : Y \to X$ such that $\pi \circ \psi = \varphi \circ \pi$.  In this situation we say that $(X,\varphi)$ is a \emph{factor} of $(Y,\psi)$.  If $\pi$ is in addition a homeomorphism then we say that $(X,\varphi)$ and $(Y, \psi)$ are \emph{(topologically) conjugate}.
\end{definition}
\begin{definition}\label{def:expansiveness}
A dynamical system $(X,\varphi)$ is called \emph{expansive} if there is a constant $\varepsilon_X > 0$, called the \emph{expansiveness constant} of $(X,\varphi)$, such that for any $x,y \in X$, \[ d(\varphi^n(x),\varphi^n(y)) \leq \varepsilon_X \text{ for all } n \in \mathbb{Z} \] implies $x = y$.
\end{definition}
\begin{example}
If $(X, \varphi)$ is a shift space, then it is expansive. In fact, every zero dimensional expansive system is topologically conjugate to a shift space, see \cite[Theorem 2.2.8]{MR3243636} and \cite{katok95}.
\end{example}
We will see many other examples of expansive dynamical systems. For an example of a non-expansive system, we have the following.
\begin{example}
Suppose that $(X, \varphi)$ is a dynamical system with $\varphi$ an isometry, meaning that
\[ d(\varphi(x_1), \varphi(x_2))=d(x_1, x_2)  \hbox{ for each }x_1, x_2  \in X. \] 
Then $(X, \varphi)$ is not expansive.
\end{example}
\begin{definition}\label{def:non-wandering}
Let $(X,\varphi)$ be a dynamical system.  We say that $x \in X$ is \emph{non-wandering} if for every neighborhood $U$ of $x$ there exists an $n > 0$ such that $\varphi^n(U) \cap U \neq \emptyset$.  We let the \emph{non-wandering set} be the set $\Omega(X,\varphi) = \{ x \in X \mid x \text{ is non-wandering} \}$.  If every $x \in X$ is non-wandering we say that $(X,\varphi)$ itself is \emph{non-wandering}.
\end{definition}
If $X$ is a compact metric space and $\varphi$ is a homeomorphism then $\Omega(X,\varphi)$ is in fact a non-empty closed subset of $X$ which is invariant under $\varphi$, making $\left(\Omega(X,\varphi), \varphi|_{\Omega(X,\varphi)}\right)$ a non-wandering dynamical system \cite{walters1981}.

\begin{definition}\label{def:periodic}
Let $(X,\varphi)$ be a dynamical system.  A \emph{periodic point} is a point $p \in X$ such that $\varphi^n(p) = p$ for some $n \geq 1$.  If $n = 1$ then $p$ is called a \emph{fixed point}.  We define the sets \[ \Per_n(X,\varphi) = \{ p \in X \mid \varphi^n(p) = p \} \] for each $n \geq 1$.  We also define the union of these, \[ \Per(X,\varphi) = \bigcup_{n \geq 1} \Per_n(X,\varphi) \,, \] which is the collection of all periodic points in $(X,\varphi)$.
\end{definition}

Note that a periodic point is clearly non-wandering, so we have that $\Per(X,\varphi) \subseteq \Omega(X,\varphi)$.  We will find that many expansive dynamical systems have a dense set of periodic points.  An interesting fact about expansive systems is that $\Per_n(X,\varphi)$ must be finite for any $n$.  We will also use the fact that if $(X,\varphi)$ and $(Y,\psi)$ are dynamical systems and $\pi : Y \to X$ is a factor map, then $\pi\parens{\Per(Y,\psi)} \subseteq \Per(X,\phi)$.  This follows because for $q \in \Per_n(Y,\psi)$ we have $\phi^n(\pi(q)) = \pi(\psi^n(q)) = \pi(q)$.

In addition to the above properties, we can also impose some global conditions on the recurrent behavior of a dynamical system.
\begin{definition} \label{def:Irreducible}
A dynamical system $(X,\varphi)$ is called \emph{irreducible} if for any ordered pair of non-empty open sets $U, V \subseteq X$ there is an $n > 0$ such that $\varphi^n(U) \cap V \neq \emptyset$.
\end{definition}
If $X$ is a compact metric space and $\varphi$ is a homeomorphism, then irreducibility for a non-wandering dynamical system is equivalent to the existence of a dense orbit in $X$, that is there is some $x \in X$ such that $\overline{\{ \varphi^n(x) \mid n \in \mathbb{Z} \}} = X$ \cite{walters1981}.  Lastly, we have the following definition.

\begin{definition}
A dynamical system $(X,\varphi)$ is called \emph{mixing} if for any ordered pair of non-empty open sets $U, V \subseteq X$ there is an $N > 0$ such that $\varphi^n(U) \cap V \neq \emptyset$ for all $n \geq N$.  
\end{definition}

Note that, of the above three global recurrence properties, we have the following implications \[ (X,\varphi) \text{ is mixing } \implies (X,\varphi) \text{ is irreducible } \implies (X,\varphi) \text{ is non-wandering } \] but the converse of each implication above is not necessarily true.  


There are several equivalence relations that we will use to understand the asymptotic behavior of dynamical systems.  

\begin{definition}
Let $(X,\varphi)$ be a dynamical system, then we have the following equivalence relations on $X$.
\begin{enumerate}[(i)]
    \item $x \sim_\text{s} y$ if and only if $\displaystyle \lim_{n\to\infty} d(\varphi^n(x), \varphi^n(y)) = 0$,
    \item $x \sim_\text{u} y$ if and only if $\displaystyle \lim_{n\to\infty} d(\varphi^{-n}(x), \varphi^{-n}(y)) = 0$, and
    \item $x \sim_\text{h} y$ if and only if $x \sim_\text{s} y$ and $x \sim_\text{u} y$.
\end{enumerate}
If $x \sim_\text{s} y$ we say that $x$ and $y$ are \emph{stably equivalent}, likewise if $x \sim_\text{u} y$ we say that $x$ and $y$ are \emph{unstably equivalent}.  If $x \sim_\text{h} y$ we say that $x$ and $y$ are \emph{homoclinic}.  We denote the stable, unstable, and homoclinic equivalence classes of $x$ as $X^\text{s}(x)$, $X^\text{u}(x)$, and $X^\text{h}(x)$, respectively.  If $P \subseteq X$ is a finite set of periodic points, then we will also denote $X^\text{s}(P)$ as the set \[ X^\text{s}(P) = \bigsqcup_{p \in P} X^\text{s}(p) \] and likewise for $X^\text{u}(P)$.  Note that by Lemma \ref{lem:periodic-h-implies-equals}, $X^\text{s}(p)$ and $X^\text{s}(q)$ are disjoint as sets if $p \neq q$.
\end{definition}

These relations tend to only be interesting if there is some chaotic behavior in $(X,\varphi)$. For example if $\varphi$ is an isometry (so that $(X, \varphi)$ is not expansive) then $x \sim_s y$ if and only if $x = y$.  We will use these relations to study the asymptotic behavior in expansive dynamical systems.  Below is a simple application of expansiveness which will become useful later.

\begin{lemma}\label{lem:periodic-h-implies-equals}
If $(X,\varphi)$ is an expansive dynamical system and $p , q \in \Per(X,\varphi)$ are periodic points such that either $p \sim_\text{s} q$ or $p \sim_\text{u} q$, then $p = q$.  In particular $p \sim_\text{h} q$ implies $p = q$.
\end{lemma}

We refer the reader to \cite{AndrewPhDThesis} for a proof.

\subsection{Smale spaces} \label{Section-SmaleSpaces}

\begin{definition} \label{SmaSpaDef}
A Smale space is a metric space $(X, d)$ along with a homeomorphism $\varphi: X\rightarrow X$ with the following additional structure: there exists global constants $\epsilon_X>0$ and $0< \lambda < 1$ and a continuous map, called the bracket map, 
\[
[ \ \cdot \  , \ \cdot \ ] :\{(x,y) \in X \times X : d(x,y) \leq \epsilon_X\}\to X
\]
such that the following axioms hold
\begin{itemize}
\item[B1] $\left[ x, x \right] = x$;
\item[B2] $\left[x,[y, z] \right] = [x, z]$ when both sides are defined;
\item[B3] $\left[[x, y], z \right] = [x,z]$ when both sides are defined;
\item[B4] $\varphi[x, y] = [ \varphi(x), \varphi(y)]$ when both sides are defined;
\item[C1] For $x,y \in X$ such that $[x,y]=y$, $d(\varphi(x),\varphi(y)) \leq \lambda d(x,y)$;
\item[C2] For $x,y \in X$ such that $[x,y]=x$, $d(\varphi^{-1}(x),\varphi^{-1}(y)) \leq \lambda d(x,y)$.
\end{itemize}
We denote a Smale space simply by $(X,\varphi)$.
\end{definition}

Examples of Smale spaces and an introduction to their basic properties can be found in \cite{putnamNotes19}. In particular, every Smale space is expansive. However, there are expansive systems that are not Smale spaces. For example, a shift space is a Smale space if and only if it is a shift of finite type, see \cite[Theorem 2.2.8]{MR3243636}. For more of shift of finite type, see \cite{lindmarcus}. We discuss one explicit example for completeness.

\begin{example}\label{ex:golden-mean-shift}
Consider the shift space $X \subseteq \{\word{a}, \word{e}, \word{f} \}^\mathbb{Z}$ given by the following graph $G$:
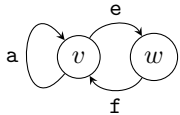
\begin{center}
\begin{tikzpicture}
\tikzset{every loop/.style={looseness=6, in=130, out=230}}
\pgfmathsetmacro{\S}{0.9}

\node[shape=circle, draw=black] (A) at (0,0) {$v$};
\node[shape=circle, draw=black] (B) at (1,0) {$w$};
\path[->,>=stealth] (A) edge[loop left] node[scale=\S, left] {$\word{a}$} (A);
\path[->,>=stealth] (A) edge[out=60, in=120] node[scale=\S, above] {$\word{e}$} (B);
\path[->,>=stealth] (B) edge[out=240, in=300] node[scale=\S, below] {$\word{f}$} (A);
\end{tikzpicture}    
\end{center}
That is to say that the elements are bi-infinite sequences of paths of edges from the graph $G$; the self-homeomorphism is the shift map, see \cite[Section 2.3]{lindmarcus} for more details. 

This is a shift space called the \emph{golden mean shift}.  It is conjugate to the shift of finite type $X_F \subseteq \{\word{0}, \word{1}\}^\mathbb{Z}$ with the set of forbidden words $F = \{ \word{11} \}$.  To see this observe that if we construct the edge shift for $X_F$ with $N = 2$, we obtain the above graph with the labeling $\word{a} = \word{00}$, $\word{e} = \word{01}$, and $\word{f} = \word{10}$ --- see \cite[Section 2.3]{lindmarcus} for more details.
\end{example}

\subsection{Finitely Presented Systems}\label{sec:finitely-presented-systems}

The definition of finitely presented systems is due to Fried \cite{fried1987}.  There are several equivalent characterizations of finitely presented systems in \cite{fried1987}, but we will use the following characterization.  An expansive dynamical system $(X,\varphi)$ is called \emph{finitely presented} if it is a factor of a shift of finite type.  As a class of dynamical systems, finitely presented systems include Smale spaces.  The shift spaces which are finitely presented systems are exactly the sofic shifts, see \cite{fried1987} for details. 

In the next example, we discuss a particular sofic shift. Our discussion summarizes a more detailed discussion of this shift in \cite{lindmarcus}, see in particular Examples 1.2.4, 1.5.6, and 2.1.5 in \cite{lindmarcus}.

\begin{example} \label{even-shift}
One explicit example of a sofic shift is the even shift. It is a nice example to have in mind when considering the definitions and theorems of the paper. The even shift is the sofic shift space in the alphabet $\{\word{0}, \word{1}\}$ whose elements do not contain any of the finite words in the set $F = \{\word{10}^{2k+1}\word{1} \mid k \geq 0\}$.  That is to say that the elements of the even shift are bi-infinite binary sequences which have an even number of consecutive zeros between any two ones.  We will denote the even shift as $X_\ev$.  We can see that the even shift is not a shift finite type since there is no upper bound on the size of the words that we must forbid, that is to say the countable set $F$ is a minimal set of forbidden words for $X_\ev$.  However, the even shift is indeed sofic since it admits the following graph presentation.

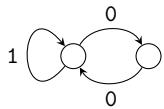
\vspace{1em}
\begin{center}
\begin{tikzpicture}
\tikzset{every loop/.style={looseness=10, in=130, out=230}}
\pgfmathsetmacro{\S}{0.9}

\node[shape=circle, draw=black] (A) at (0,0) {};
\node[shape=circle, draw=black] (B) at (1,0) {};
\path[->,>=stealth] (A) edge[loop left] node[scale=\S, left] {$\word{1}$} (A);
\path[->,>=stealth] (A) edge[out=60, in=120] node[scale=\S, above] {$\word{0}$} (B);
\path[->,>=stealth] (B) edge[out=240, in=300] node[scale=\S, below] {$\word{0}$} (A);
\end{tikzpicture}    
\end{center}
In this case, the edge shift $X_G$ coming from the unlabeled graph above is exactly the golden mean shift discussed in Example \ref{ex:golden-mean-shift}.  In other words, the even shift is a factor of the golden mean shift.

It is important to note that the factor map $\pi : X_G \to X$ may not be injective.  For example in the even shift there are two elements of $X_G$ that get mapped to the point $\overline{\word{0}} = \dots \word{00.000} \dots$.  However in this case this is the only such point since as soon there is a $\word{1}$ in a sequence $x \in X$, we know exactly which element in $X_G$ must correspond to it.  This follows from the fact that there is only one vertex with an arrow labeled $\word{1}$ pointing to it.  Hence $\pi$ is injective on almost all points in $X_G$ except the two points that get mapped to $\overline{\word{0}}$.
\end{example}

We have the following theorem which says that all expansive homeomorphisms of orientable surfaces are finitely presented, see \cite{hiraide87, hiraide87-2, lewowicz89, artigue2008local}.  We have adapted the language of this theorem.

\begin{theorem}{\cite{hiraide87, hiraide87-2, lewowicz89, Lewowicz08, artigue2008local}} \label{thm:surfaces}
Let $\varphi : M \to M$ be an expansive homeomorphism of a compact, connected, oriented, boundaryless surface $M$.  Then
\begin{enumerate}[(i)]
    \item $M$ cannot be $S^2$,
    \item if $M$ is the 2-torus then $(M, \varphi)$ is conjugate to a hyperbolic toral automorphism (\ie it is a Smale space), and
    \item if the genius of $M$ is larger than 1, then $(M, \varphi)$ is finitely presented but not a Smale space.
\end{enumerate}
\end{theorem}
There are many more examples of finitely presented systems that are not Smale spaces, see \cite{fried1987} for details.

\section{Local Product Structure}\label{sec:local-stable-unstable}

In this section we will discuss the construction of the local stable and unstable sets of an element of an expansive dynamical system.  We will also discuss a result of Fried about the existence of a particular metric on expansive systems, see \cite[Lemma 2]{fried1987}.  This metric shows the hyperbolic behavior of expansive systems similar to Axioms C1 and C2 for Smale spaces.  Lastly we will show that the stable (unstable) equivalence class of an element in an expansive dynamical system can be written as a union of local stable (unstable) sets.

Let $(X, \varphi)$ be an expansive dynamical system with expansiveness constant $\varepsilon_X > 0$.  In the rest of this section we follow \cite{fried1987}.
\begin{definition}\label{def:loc-stable-unstable-sets}
For each $x \in X$ and $\varepsilon > 0$ we define the following subsets called, respectively, the \emph{local stable} and \emph{local unstable} sets of $x$.
\begin{align*}
    X^\text{s}(x, \varepsilon) &= \left\{ y \in X \mid d\left(\varphi^n(x),\varphi^n(y)\right) \leq \varepsilon \text{ for all } n \geq 0 \right\} \\
    X^\text{u}(x, \varepsilon) &= \left\{ y \in X \mid d\left(\varphi^{-n}(x),\varphi^{-n}(y)\right) \leq \varepsilon \text{ for all } n \geq 0 \right\}
\end{align*}
Furthermore, one can easily show from this definition that
\begin{align*}
    \varphi^{-N}\left(X^\text{s}\left(\varphi^{N}(x), \varepsilon\right)\right) &= \left\{ y \in X \mid d\left(\varphi^n(x),\varphi^n(y)\right) \leq \varepsilon \text{ for all } n \geq N \right\} \text{ , and } \\
    \varphi^{N}\left(X^\text{u}\left(\varphi^{-N}(x), \varepsilon\right)\right) &= \left\{ y \in X \mid d\left(\varphi^{-n}(x),\varphi^{-n}(y)\right) \leq \varepsilon \text{ for all } n \geq N \right\}
\end{align*}
for any $N \in \Z$.
\end{definition}
The next lemma is a well-known application of expansiveness; the proof is omitted.
\begin{lemma}\label{lem:bracket-unique-intersection}
For $\displaystyle 0 < \varepsilon \leq \frac{\varepsilon_X}{2}$, the intersection $X^\text{s}(x, \varepsilon) \cap X^\text{u}(y, \varepsilon)$ consists of at most one point in $X$.
\end{lemma}

\subsection{The Bracket Map}

\begin{definition}\label{def:bracket-map}
For $0 < \varepsilon \leq \frac{\varepsilon_X}{2}$ we define the set
\[ D_\varepsilon = \{ (x, y) \in X \times X \mid X^\text{s}(x, \varepsilon) \cap X^\text{u}(y, \varepsilon) \neq \emptyset \} \]
and a map $[-,-]: D_\varepsilon \to X$, called the \emph{bracket map}, such that $[x,y] \in X^\text{s}(x, \varepsilon) \cap X^\text{u}(y, \varepsilon)$.  By Lemma \ref{lem:bracket-unique-intersection} this map is well-defined.
\end{definition}
\begin{remark}
It is well-known that if $(X, \varphi)$ is a Smale space, then the definition of the bracket map in the previous definition agrees with the one in the definition of a Smale space.
\end{remark}

The following lemma is due to Fried \cite{fried1987}.  One can find a detailed proof in \cite{AndrewPhDThesis}.

\begin{lemma}{\cite{fried1987}}
The map $[-,-]$ is continuous and $D_\varepsilon$ is closed in $X \times X$.  Additionally, the set $\Delta_X = \{(x,x) \mid x \in X\}$ is contained in $D_\varepsilon$, and $[x,x] = x$ for all $x \in X$.
\end{lemma}

\subsection{Adapted Metric}\label{sec:adapted-metric}

For any expansive dynamical system there exists a metric $d$, which we will call an \emph{adapted metric} \cite[Lemma 2]{fried1987}, and constants $\eta > 0$, $0 < \lambda < 1$ such that $d$ is compatible with the topology on $X$ and
\begin{align*}
    d(\varphi(x), \varphi(y)) &\leq \lambda d(x, y) \text{ for all } y \in X^\text{s}(x, \eta) \text{ , and} \\
    d(\varphi^{-1}(x), \varphi^{-1}(y)) &\leq \lambda d(x, y) \text{ for all } y \in X^\text{u}(x, \eta) \,.
\end{align*}

Furthermore, this metric can additionally be chosen so that both $\varphi$ and $\varphi^{-1}$ are Lipschitz for some Lipschitz constant $K > 1$, see \cite[Lemma 2]{fried1987}.  Although this choice of metric is not necessarily unique, for the rest of the paper we will assume that if $(X,\varphi)$ is an expansive dynamical system then the metric on $X$ is always an adapted metric as described above.  Hence we will often make reference to ``the'' adapted metric.

Using the adapted metric, we can prove the following useful lemma.

\begin{lemma}\label{lemma:bracket-epsilon-continuity}
Fix $0 < \varepsilon_0 \leq \eta$ and assume there is $(x,x) \in \interior(D_{\varepsilon_0})$.  Then for any $\varepsilon > 0$ there exists a $\delta > 0$ such that if $d(x,y) < \delta$ and $d(x,z) < \delta$, then $[y,z]$ is defined and satisfies $[y,z] \in X^\text{s}(y, \varepsilon) \cap X^\text{u}(z, \varepsilon)$.
\end{lemma}
\begin{proof}
Assume that $\varepsilon \leq \eta$.  Since $(x,x) \in \interior(D_{\varepsilon_0})$, and by the continuity of $[-,-]$, there is $\delta > 0$ such that $d(x,y) < \delta$ and $d(x,z) < \delta$ imply $[y,z]$ is defined and additionally that $d(y, [y,z]) \leq \varepsilon$ and $d(z, [y,z]) \leq \varepsilon$.  Since $d$ is the adapted metric and $[y,z] \in X^\text{s}(y, \varepsilon_0)$, we have \[ d(\varphi^n(y), \varphi^n([y,z])) \leq \lambda^n d(y, [y, z]) \leq \lambda^n \varepsilon \leq \varepsilon \] for all $n \geq 0$.  Similarly we have $d(\varphi^{-n}(z),\varphi^{-n}([y,z])) \leq \varepsilon$ for all $n \geq 0$.  Hence $[y,z] \in X^\text{s}(y, \varepsilon) \cap X^\text{u}(z, \varepsilon)$.
\end{proof}

\subsection{Global Stable/Unstable Sets}\label{sec:global-stable-unstable-sets}

For the sake of Lemma \ref{lem:stable-unstable-inductive-limit-top}, we introduce the following sets
\begin{align*}
    X^\text{s}_<(x, \varepsilon) &= \left\{ y \in X \mid d\left(\varphi^n(x),\varphi^n(y)\right) < \varepsilon \text{ for all } n \geq 0 \right\} \\
    X^\text{u}_<(x, \varepsilon) &= \left\{ y \in X \mid d\left(\varphi^{-n}(x),\varphi^{-n}(y)\right) < \varepsilon \text{ for all } n \geq 0 \right\}
\end{align*}

\begin{lemma}\label{lem:stable-unstable-inductive-limit-top}
Let $(X,\varphi)$ be an expansive dynamical system and let $x \in X$.  Then for any $\varepsilon > 0$ the stable and unstable equivalence classes satisfy
\begin{enumerate}[(i)]
    \item $\displaystyle X^\text{s}(x) = \bigcup_{N \geq 0} \varphi^{-N}\left(X^\text{s}_<\left(\varphi^{N}(x), \varepsilon\right)\right)$, and
    \item $\displaystyle X^\text{u}(x) = \bigcup_{N \geq 0} \varphi^{N}\left(X^\text{u}_<\left(\varphi^{-N}(x), \varepsilon\right)\right)$,
\end{enumerate}
respectively.  Hence we will often refer to the equivalence classes $X^\text{s}(x)$ and $X^\text{u}(x)$ as the \emph{global stable set} and \emph{global unstable set} of $x$, respectively.
\end{lemma}
\begin{proof}
In the stable case, this can easily be seen from definition of stable equivalence and the fact discussed above that \[ \varphi^{-N}\left(X^\text{s}_<\left(\varphi^{N}(x), \varepsilon\right)\right) = \left\{ y \in X \mid d\left(\varphi^n(x),\varphi^n(y)\right) < \varepsilon \text{ for all } n \geq N \right\} \] for all $N \geq 0$.
\end{proof}

When $p$ is a periodic point, we topologize $X^\text{s}(p)$ and $X^\text{u}(p)$ with the inductive limit topology coming from Lemma \ref{lem:stable-unstable-inductive-limit-top}, see \cite[Section 4.1]{thomsen2010c}.  Note that with this topology $X^\text{s}(p)$ and $X^\text{u}(p)$ are both locally compact and Hausdorff.

\section{Synchronizing Dynamical Systems} \label{SecSynDS}

\subsection{Synchronizing Elements of an Expansive Dynamical System}

Suppose $(X,\varphi)$ is an expansive dynamical system. Recall from Section \ref{sec:local-stable-unstable} we will be using the adapted metric and that for $0 < \varepsilon \leq \frac{\varepsilon_X}{2}$ we define the set
\[ D_\varepsilon = \{ (x, y) \in X \times X \mid X^\text{s}(x, \varepsilon) \cap X^\text{u}(y, \varepsilon) \neq \emptyset \} \,. \]
Additionally we defined the bracket map $[-,-]:D_\varepsilon \to X$ by $[x,y] \in X^\text{s}(x,\varepsilon) \cap X^\text{u}(y,\varepsilon)$.  As was mentioned in the previous section, these definitions are due to Fried \cite{fried1987}.

\begin{lemma}\label{lem:epsilon-naught-existence}
Suppose $(X,\varphi)$ is an expansive dynamical system.  There exists $\varepsilon_0 > 0$ such that for $\displaystyle 0 < \varepsilon, \varepsilon' \leq \varepsilon_0$, $(x,x)$ is in the interior of $D_\varepsilon$ if and only if it is in the interior of $D_{\varepsilon'}$.
\end{lemma}
\begin{proof}
Let $\displaystyle \varepsilon_0 = \min\left\{\eta, \frac{\varepsilon_X}{2}\right\}$ where $\eta$ is as in Section \ref{sec:adapted-metric}.  Without loss of generality assume $(x,x)$ is in the interior of $D_\varepsilon$, then by continuity we may find $\delta > 0$ such that $\delta \leq \varepsilon'/2$ and small enough that $d(x,y) < \delta$ and $d(x,z) < \delta$ implies $[y,z]$ is defined and $d(x, [y,z]) < \varepsilon'/2$, see Lemma \ref{lemma:bracket-epsilon-continuity}.  Hence \[ d(y,[y,z]) \leq d(y, x) + d(x, [y, z]) < \varepsilon' \] and since $d$ is the adapted metric and $[y,z] \in X^\text{s}(y, \varepsilon)$, for $n \geq 0$, we have \[ d(\varphi^n(y), \varphi^n([y,z])) \leq \lambda^n d(y,[y,z]) < \lambda^n \varepsilon' \leq \varepsilon' \,, \] meaning $[y,z] \in X^\text{s}(y, \varepsilon')$.  Similarly $[y,z] \in X^\text{u}(z, \varepsilon')$, so we can conclude that $(y,z) \in D_{\varepsilon'}$ and $(x,x)$ is in the interior of $D_{\varepsilon'}$.
\end{proof}

\begin{definition}\label{def:synchronizing-point}
Suppose $(X,\varphi)$ is an expansive dynamical system.  A point $x \in X$ is called \emph{synchronizing} if $(x,x)$ is in the interior of $D_\varepsilon$ for some $\displaystyle 0 < \varepsilon \leq \varepsilon_0$, where $\varepsilon_0$ is as in Lemma \ref{lem:epsilon-naught-existence}.  We will denote the subset of synchronizing points $X_\sync$.
\end{definition}

By Lemma \ref{lem:epsilon-naught-existence}, this definition of synchronizing does not depend on the choice of $\varepsilon$.  We now show the crucial property of synchronizing points, which is that they are exactly the points in an expansive dynamical system having local product structure.

\begin{prop}\label{rectangle-nbhd}
Let $\varepsilon_0$ be as in Lemma \ref{lem:epsilon-naught-existence}.  Then an element $x \in X$ is synchronizing if and only if for each $0 < \varepsilon \leq \varepsilon_0$ there is a $\delta_x > 0$ such that for $0 < \delta \leq \delta_x$ we have
\begin{enumerate}[(i)]
    \item $X^\text{u}(x, \delta) \times X^\text{s}(x, \delta) \subseteq D_\varepsilon$, and
    \item $[-,-]$ restricted to $X^\text{u}(x, \delta) \times X^\text{s}(x, \delta)$ is a homeomorphism onto its image, which is a neighborhood of $x$.
\end{enumerate}
\end{prop}
\begin{proof}
First we assume $x \in X$ is synchronizing.  Fix $0 < \varepsilon \leq \varepsilon_0$.  Since for any $\displaystyle 0 < \varepsilon \leq \varepsilon_0$ we know $(x,x)$ is also on the interior of $D_{\varepsilon/2}$, there is an $\varepsilon'$ such that if $d(x,y) < \varepsilon'$ and $d(x,z) < \varepsilon'$ then $(y,z), (z, y) \in D_{\varepsilon/2}$.  Let $0 < \delta_x < \varepsilon'$ be such that $d(x,y) \leq \delta_x$ and $d(x,z) \leq \delta_x$ implies $d(x,[y,z]) < \varepsilon'$.  As a consequence of this, we have that $[y,z]$, $[x,z]$, $[x,y]$, $[x,[y,z]]$, and $[[y,z],x]$ are all defined for $D_{\varepsilon/2}$.  In particular it is true that $X^\text{u}(x, \delta) \times X^\text{s}(x, \delta) \subseteq D_\varepsilon$ for any choice of $0 < \delta \leq \delta_x$.

We will also need the following fact.  Since $[y,z] \in X^\text{s}(y, \varepsilon/2)$ and $[[y,z],x] \in X^\text{s}([y,z], \varepsilon/2)$, we have
\[ d(\varphi^n([[y,z],x]), \varphi^n(y)) \leq d(\varphi^n([[y,z],x]), \varphi^n([y,z])) + d(\varphi^n([y,z]), \varphi^n(y)) \leq \varepsilon \]
for all $n \geq 0$.  This means $[[y,z],x] \in X^\text{s}(y, \varepsilon) \cap X^\text{u}(x, \varepsilon)$ and so $[[y,z],x] = [y,x]$ by Lemma \ref{lem:bracket-unique-intersection}.  By similar reasoning we also have $[x,[y,z]] = [x,z]$.

Fix $0 < \delta \leq \delta_x$.  Let $[-,-]_x$ denote the bracket map $[-,-]$ restricted to $X^\text{u}(x, \delta) \times X^\text{s}(x, \delta)$, and let $R$ denote the image of $[-,-]_x$.  First we remark that if $y \in X^\text{u}(x, \delta)$ then $y \in X^\text{s}(y, \varepsilon) \cap X^\text{u}(x, \varepsilon)$ and so $y = [y,x]$, and similarly $z \in X^\text{s}(x, \delta)$ implies $z = [x,z]$.  Define a function $h_x : R \to X^\text{u}(x, \delta) \times X^\text{s}(x, \delta)$ by $h_x(y) = ([y,x],[x,y])$.  To see that $h_x$ is well-defined on $R$ and indeed a right inverse of $[-,-]_x$, we compute
\[ h_x([y,z]) = ([[y,z], x], [x, [y, z]]) = ([y,x], [x,z]) = (y, z) \,. \] Note also that the continuity of $h_x$ follows from the continuity of $[-,-]$.  Lastly we must check that $h_x$ is also a left inverse of $[-,-]_x$ by computing
\[ ([-,-] \circ h_x)(y) = [[y,x],[x,y]] = y \,. \]  The last equality follows from the fact that $y \in X^\text{s}([y,x], \varepsilon) \cap X^\text{u}([x,y], \varepsilon)$.  We conclude that $[-,-]_x : X^\text{u}(x, \delta) \times X^\text{s}(x, \delta) \to R$ is indeed a homeormorphism.

Furthermore it is clear that $R$ is a neighborhood of $x$ since for $y$ close enough to $x$ we have that $[y,x]$ and $[x,y]$ are defined and in $X^\text{u}(x, \delta)$ and $X^\text{s}(x, \delta)$ respectively, implying $y$ is in $R$.

Conversely, if for any $0 < \varepsilon \leq \varepsilon_0$ there is a $\delta_x$ such that $[-,-] : X^\text{u}(x, \delta) \times X^\text{s}(x, \delta) \to R$ is a homeomorphism for any $0 < \delta \leq \delta_x$, it is clear that $X^\text{u}(x, \delta) \times X^\text{s}(x, \delta) \subseteq D_\varepsilon$ is a neighborhood of $(x,x)$ in $D_\varepsilon$, hence $x$ is synchronizing.
\end{proof}

We will often call the neighborhood $R$ in the proof of Proposition \ref{rectangle-nbhd} a \emph{product neighborhood} (or \emph{rectangular neighborhood}) of $x$.  Alternatively we will say that $x$ has \emph{local product structure}.




\subsection{The Definition of Synchronizing Dynamical Systems}

\begin{definition}\label{def:synchronizing-general}
An expansive dynamical system $(X,\varphi)$ is called \emph{synchronizing} if it is irreducible and it has at least one synchronizing point.  Let the set of synchronizing points in $X$ be denoted $X_\text{sync}$.  Hence an irreducible expansive system $(X,\varphi)$ is called synchronizing if $X_\text{sync} \neq \emptyset$.
\end{definition}

\begin{lemma}\label{lemma:sync-invariant-under-map}
Let $(X,\varphi)$ be an expansive system and $x \in X$.  Then $x$ is synchronizing if and only if $\varphi(x)$ is synchronizing.
\end{lemma}
\begin{proof}
Assume $x$ is synchronizing with $\varepsilon$, $\delta_x$, and $R$ as in Proposition \ref{rectangle-nbhd}, and furthermore that $\delta_x$ is also small enough that $\varphi(X^\text{u}(z,\delta_x)) \subseteq X^\text{u}(\varphi(z), \varepsilon)$ for all $z \in R$.  Let $y = \varphi(x)$, and let $\delta_y > 0$ be small enough that $(z,z') \in X^\text{u}(y,\delta_y) \times X^\text{s}(y,\delta_y)$ implies $(\varphi^{-1}(z), \varphi^{-1}(z')) \in X^\text{u}(x, \delta_x) \times X^\text{s}(x, \delta_x)$.  Then
\begin{align*}
    \varphi\left([\varphi^{-1}(z), \varphi^{-1}(z')]\right) &\in \varphi\left( X^\text{s}(\varphi^{-1}(z), \delta_x) \cap X^\text{u}(\varphi^{-1}(z'), \delta_x) \right) \\
    &= \varphi\left(X^\text{s}(\varphi^{-1}(z), \delta_x)\right) \cap \varphi\left(X^\text{u}(\varphi^{-1}(z'), \delta_x)\right) \\
    &\subseteq X^\text{s}(z, \varepsilon) \cap X^\text{u}(z', \varepsilon)
\end{align*}
meaning $[z,z'] = \varphi\left([\varphi^{-1}(z), \varphi^{-1}(z')]\right)$, so $y$ is synchronizing.  Note that $\varphi^{-1}(z') \in R$ since $[x,\varphi^{-1}(z')] = \varphi^{-1}(z')$, so the final step above is justified.  That $\varphi(x)$ is synchronizing implies $x$ is synchronizing is proved in a similar way, we omit the details.
\end{proof}

\begin{prop}\label{theorem:sync-points-dense}
If $(X,\varphi)$ is synchronizing then $X_\text{sync}$ is a dense open set in $X$.
\end{prop}
\begin{proof}
The set of synchronizing points, $X_\sync$, is open by Proposition \ref{rectangle-nbhd}, and non-empty by definition.  Since $X$ is irreducible it has a point with a dense orbit, see the discussion after Definition \ref{def:Irreducible}. Since $X_\sync$ contains this orbit by Lemma \ref{lemma:sync-invariant-under-map}, $X_\sync$ is also dense.
\end{proof}

\subsection{Periodic Points in Synchronizing Systems}
Our goal in this section is to prove that the set of periodic points is dense for a synchronizing system.  The proof of Theorem \ref{theorem:periodic-existence} is new, as for Smale spaces one uses the shadowing property which no longer holds in the generality of synchronizing systems, see \cite[Theorem 1]{sakai01}.  We instead show that the existence of periodic points is a direct consequence of the existence of rectangular neighborhoods.

First we must prove the following lemma adapted from an argument by Sakai in \cite[Theorem 1]{sakai01}.  One can find a detailed proof in \cite{AndrewPhDThesis}.  In Lemma \ref{lem:bracket-lipschitz} note that the constant $C$ only depends on the constants $\lambda$ and $K$ associated to the adapted metric.

\begin{lemma}{\cite{sakai01}}\label{lem:bracket-lipschitz}
Let $(X,\varphi)$ be an expansive dynamical system.  Then there is a constant $C > 0$ such that for any synchronizing point $x \in X$ there exists a neighborhood $U$ of $x$ where $d(y, [y,z]) \leq C d(y, z)$ and $d(z, [y,z]) \leq C d(y, z)$ for all $y, z \in U$.
\end{lemma}

\begin{theorem}\label{theorem:periodic-existence}

Let $(X,\varphi)$ be an non-wandering expansive dynamical system and $x \in X$ a synchronizing point.  Then for any neighborhood $U$ of $x$ there is a periodic point in $U$.
\end{theorem}
\begin{proof}
By Lemma \ref{lem:bracket-lipschitz}, there is a constant $C$ and neighborhood $\widetilde{U}$ of $x$ such that $d(y, [y,z]) \leq C d(y, z)$ and $d(z, [y,z]) \leq C d(y, z)$ for all $y, z \in \widetilde{U}$.

Fix $\displaystyle 0 < \widetilde{C} < \frac{1}{2+4C}$.  Let $R$ be a product neighborhood of $x$ and let $\delta > 0$ such that $\overline{B_\delta(x)} \subseteq R \cap U \cap \widetilde{U}$.  Let $\displaystyle \delta' = \widetilde{C}\delta$ and $V = B_{\delta'}(x)$, noting that $\displaystyle \delta' < \frac{\delta}{2}$, so $V \subseteq B_\delta(x)$.  Since $x$ is non-wandering, there is an $n$ such that $\varphi^n(V) \cap V \neq \emptyset$.  Furthermore we assume that $n$ is such that \[ \lambda^n < \frac{1}{4C} \] and we set $\displaystyle \beta = 2C\lambda^n < \frac{1}{2}$.

Hence there is a $y \in V$ such that $\varphi^n(y) \in V$.  Set $z_0 = [y, \varphi^n(y)]$, which is well-defined since $y, \varphi^n(y) \in V \subseteq R$.  Then from Lemma \ref{lem:bracket-lipschitz} we have \[ z_0 \in X^\text{s}(y, 2C\delta') \cap X^\text{u}(\varphi^n(y), 2C\delta') \,. \label{eq:periodic-induction-z0} \tag{$\ast$} \]
Using this we get an estimate on $d(x,z_0)$,
\begin{align*}
    d(x, z_0) &\leq d(x, y) + d(y, z_0) \\
    &< \delta' + C d(y, \varphi^n(y)) \\
    &< \delta' + C \diam(V) \\
    &= (1 + 2C) \delta' \\
    &< \frac{1+2C}{2+4C}\delta \\
    &= \frac{\delta}{2}
\end{align*}
by our choice of $\widetilde{C}$.  Hence $z_0 \in B_\delta(x) \subseteq \widetilde{U}$.  It follows from \eqref{eq:periodic-induction-z0} and the adapted metric that
\begin{align*}
    \varphi^n(z_0) &\in X^\text{s}(\varphi^n(y), \beta \delta') \\
    \varphi^{-n}(z_0) &\in X^\text{u}(y, \beta \delta') \,.
\end{align*}
We then have \[ d(x, \varphi^n(z_0)) \leq d(x, \varphi^n(y)) + d(\varphi^n(y), \varphi^n(z_0)) < (1 + \beta) \delta' \,. \]  A similar computation shows we also have $\displaystyle d(x, \varphi^{-n}(z_0)) < (1 + \beta) \delta'$.  Since $(1+\beta)\delta' < \delta$ we may define $z_1 = [\varphi^{-n}(z_0), \varphi^n(z_0)]$.  Then, because $d(\varphi^{-n}(z_0), \varphi^n(z_0)) < 2(1 + \beta) \delta'$, we can apply Lemma \ref{lem:bracket-lipschitz} again to obtain \[ z_1 \in X^\text{s}(\varphi^{-n}(z_0), 2C(1+\beta)\delta') \cap X^\text{u}(\varphi^n(z_0), 2C(1+\beta)\delta') \,. \]  Lastly, we obtain the following estimate on $d(x,z_1)$,
\begin{align*}
    d(x,z_1) &\leq d(x, \varphi^n(z_0)) + d(\varphi^n(z_0), z_1) \\
    &< (1+\beta)\delta' + 2C(1+\beta)\delta' \\
    &= (1+\beta)(1+2C)\delta' \\
    &< (1+\beta)\frac{\delta}{2} \\
    &< \frac{3\delta}{4}
\end{align*}
using $\displaystyle \beta < \frac{1}{2}$ and $\displaystyle (1 + 2C)\delta' < \frac{\delta}{2}$.

Our goal with this proof is to define a sequence $\{ z_m \}_{m \geq 0}$ from which we will obtain a subsequence converging to a periodic point in $U$.  This sequence will be defined recursively by \[ z_{m + 1} = [\varphi^{-n}(z_m), \varphi^n(z_m)] \] for all $m \geq 0$, with $z_0$ and $z_1$ already constructed above.  The proof will proceed by induction as follows, with the base case already having been shown.

Fix $m \geq 1$.  Suppose \[ z_m \in X^\text{s}\big(\varphi^{-n}(z_{m-1}), 2C(\beta^{m-1} + \beta^m)\delta'\big) \cap X^\text{u}\big(\varphi^n(z_{m-1}), 2C(\beta^{m-1} + \beta^m)\delta'\big) \label{eq:periodic-induction-A} \tag{A} \] and \[ d(x, z_i) < \left( \sum_{k=1}^{i + 1} \frac{1}{2^k} \right) \delta \label{eq:periodic-induction-B} \tag{B} \] for all $0 \leq i \leq m$.  In particular, $z_m = [\varphi^{-n}(z_{m-1}), \varphi^n(z_{m-1})]$ is well-defined.  We will show that both \eqref{eq:periodic-induction-A} and \eqref{eq:periodic-induction-B} hold for $m + 1$.

From \eqref{eq:periodic-induction-A} and the adapted metric we have
\begin{align*}
    \varphi^n(z_m) &\in X^\text{s}\big(z_{m-1}, (\beta^m + \beta^{m+1}) \delta'\big) \\
    \varphi^{-n}(z_m) &\in X^\text{u}\big(z_{m-1}, (\beta^m + \beta^{m+1}) \delta'\big) \,.
\end{align*}
It then follows that
\begin{align*}
   d(x, \varphi^n(z_m)) &\leq d(x, z_{m-1}) + d(z_{m-1}, \varphi^n(z_m)) \\
   &< d(x, z_{m-1}) + (\beta^m + \beta^{m+1}) \delta' \\
   &< d(x, z_{m-1}) + \left(\frac{1}{2^m} + \frac{1}{2^{m+1}}\right) \frac{\delta}{2} \\
   &< \left( \sum_{k=1}^{m+2} \frac{1}{2^k} \right) \delta \\
   &< \delta \,.
\end{align*}
Above we are using the facts that $\displaystyle \beta < \frac{1}{2}$ and $\displaystyle \delta' < \frac{\delta}{2}$.  Similarly one can show $d(x, \varphi^{-n}(z_m)) < \delta$.  Hence $z_{m+1} = [\varphi^{-n}(z_m), \varphi^n(z_m)]$ is defined.  Since \[ d(\varphi^{-n}(z_m), \varphi^n(z_m)) \leq d(\varphi^{-n}(z_m), z_{m-1}) + d(z_{m-1}, \varphi^n(z_m)) \leq 2(\beta^m + \beta^{m+1})\delta' \] we obtain from Lemma \ref{lem:bracket-lipschitz} that \[ z_{m+1} \in X^\text{s}\big(\varphi^{-n}(z_m), 2C(\beta^m + \beta^{m+1})\delta'\big) \cap X^\text{u}\big(\varphi^n(z_m), 2C(\beta^m + \beta^{m+1})\delta'\big) \,. \]  This shows \eqref{eq:periodic-induction-A} is true for $m + 1$.  Furthermore,
\begin{align*}
   d(x, z_{m+1}) &\leq d(x, z_{m-1}) + d(z_{m-1}, \varphi^n(z_m)) + d(\varphi^n(z_m), z_{m+1}) \\
   &< d(x, z_{m+1}) + (\beta^m + \beta^{m+1}) \delta' + 2C(\beta^{m-1} + \beta^m)\delta' \\
   &= d(x, z_{m+1}) + (\beta^m + \beta^{m+1}) (1 + 2C) \delta' \\
   &< d(x, z_{m+1}) + \left(\frac{1}{2^m} + \frac{1}{2^{m+1}}\right) \frac{\delta}{2} \\
   &< \left( \sum_{k=1}^{m+2} \frac{1}{2^k} \right) \delta \,.
\end{align*}
Hence \eqref{eq:periodic-induction-B} is true for $m + 1$, and this completes the induction proof.

The sequence $\{z_m\}_{m \geq 0}$ is defined for all $m$ and also contained in $B_\delta(x)$. By compactness, there is a convergent subsequence $\{z_{m_k}\}_{k \geq 0}$ converging to some element $p$.  Furthermore since $\overline{B_\delta(x)} \subseteq U$ we have $p \in U$.  Since $\displaystyle 0 < \beta < \frac{1}{2}$, \[ d(\varphi^{-n}(p), \varphi^n(p)) = \lim_{k \to \infty} d(\varphi^{-n}(z_{m_k}), \varphi^n(z_{m_k})) \leq \lim_{k \to \infty} 2(\beta^{m_k} + \beta^{m_k+1})\delta' = 0 \,. \] This shows that $p$ is a periodic point with period $2n$.
\end{proof}

\begin{corollary}\label{cor:sync-periodic-dense}
If $(X,\varphi)$ is synchronizing then $\text{Per}(X,\varphi)$ is dense in $X$.
\end{corollary}
\begin{proof}
This follows from Theorem \ref{theorem:periodic-existence} and Proposition \ref{theorem:sync-points-dense}.
\end{proof}

We can now remark that there are expansive dynamical systems which are not synchronizing. For example, Toeplitz flows are minimal, expansive subshifts. Hence a Toeplitz flows (with infinitely many points) being minimal cannot have any periodic points and so it cannot be synchronizing, see \cite{downarowicz05} for more on Toeplitz flows.

\subsection{Covers of Expansive Systems by Rectangles}

Here we show that finitely presented systems are in fact synchronizing.  In fact we will show something more general.

Fried defines a property for expansive systems, equivalent to being finitely presented, called finitely rectangled \cite{fried1987}.  This is defined as follows: fix $\varepsilon > 0$, then $R \subseteq X$ is called a \emph{rectangle} if $R \times R \subseteq D_\varepsilon$.  We can assume that a rectangle is closed since $D_\varepsilon$ is closed.  Then, as in the proof of Proposition \ref{rectangle-nbhd}, if $x \in R$ one can show that $R$ is homeomorphic to $(X^\text{u}(x,\varepsilon) \cap R) \times (X^\text{s}(x,\varepsilon) \cap R)$ via the map $h_x(y) = ([y,x],[x,y])$.  Then an expansive system is called \emph{finitely rectangled} if $X$ is a finite union of rectangles.  

\begin{prop}\label{theorem:countably-rectangled-implies-sync}
Suppose $(X,\varphi)$ is an irreducible expansive dynamical system that is the union of countably many rectangles.  Then $(X,\varphi)$ is synchronizing.
\end{prop}
\begin{proof}
Let $\mathcal{R}$ be a countable collection of closed rectangles whose union is $X$.  By the Baire Category Theorem, there exists $R \in \mathcal{R}$ with non-empty interior.  Let $x \in \text{int}(R)$, then since $R \times R \subseteq D_\varepsilon$ we must have $(x,x) \in \text{int}(D_\varepsilon)$.  Hence $x$ is synchronizing and so $(X,\varphi)$ is synchronizing.
\end{proof}
It is natural to ask if there is a converse to the previous theorem; this is an open problem.
\begin{corollary}
If $(X,\varphi)$ is an irreducible finitely presented system then it is synchronizing.
\end{corollary}
\begin{proof}
If $(X,\varphi)$ is finitely presented then it is finitely rectangled.  The result now follows from Proposition \ref{theorem:countably-rectangled-implies-sync}.
\end{proof}
\subsection{Examples}
There are many synchronizing systems that are not finitely presented. In particular, there are many shifts space that are synchronizing, but not sofic. A number of these are considered in detail in \cite{AndrewPhDThesis}. An explicit example of a synchronizing shift that is not sofic is the context-free shift, see \cite[Example 3.1.7]{lindmarcus}. We discuss a few more examples below.
\begin{example} \label{nonSoficShift}
Let $X \subseteq \{\word{a},\word{b},\word{c}\}^\mathbb{Z}$ be the closure of the set of bi-infinite paths on the following graph.

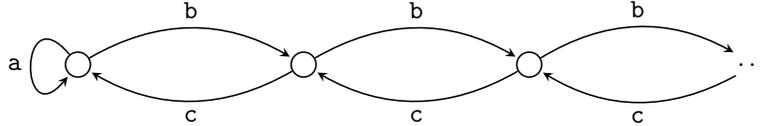
\begin{center}
\scalebox{1}{
\begin{tikzpicture}[
        > = stealth, 
        shorten > = 1pt, 
        auto,
        node distance = 3cm, 
        semithick, 
        scale = 0.5
    ]

    \tikzset{vertex/.style = {shape=circle,draw,minimum size=0.5em}}
    \tikzset{every loop/.style={looseness=10, in=230, out=130}}

    \node[vertex] (v1) {};
    \node[vertex] (v2) [right of=v1] {};
    \node[vertex] (v3) [right of=v2] {};
    \node (v4) [right of=v3] {\dots};

    \path[->] (v1) edge[loop left] node {$\word{a}$} (v1);
    \path[->] (v1) edge[bend left] node {$\word{b}$} (v2);
    \path[->] (v2) edge[bend left] node {$\word{b}$} (v3);
    \path[->] (v3) edge[bend left] node {$\word{b}$} (v4);
    \path[->] (v4) edge[bend left] node {$\word{c}$} (v3);
    \path[->] (v3) edge[bend left] node {$\word{c}$} (v2);
    \path[->] (v2) edge[bend left] node {$\word{c}$} (v1);

\end{tikzpicture}
}
\end{center}
This is a mixing synchronizing shift.  To see it is synchronizing, observe that that a word $w \in \lang(X)$ is a sychronizing word \cite{shift-paper} if and only if $\word{a} \sqsubseteq w$.  Note however $X$ is not sofic since the words $\{ \word{ab}^n \mid n \geq 0 \}$ each have distinct follower sets, see \cite[Theorem 3.2.10]{lindmarcus}.  Hence $X$ is not finitely presented.
\end{example}
\begin{proposition} \label{productProp}
Suppose that $(X_1, \varphi_1)$ is not finitely generated, but is synchronizing and $(X_2, \varphi_2)$ is synchronizing and mixing. Then $(X_1\times X_2, \varphi_1 \times \varphi_2)$ is not finitely generated, but is synchronizing. 
\end{proposition}
\begin{proof}
We must check that
\begin{enumerate}
\item $(X_1\times X_2, \varphi_1 \times \varphi_2)$ is expansive and irreducible;
\item $(X_1\times X_2, \varphi_1 \times \varphi_2)$ has a synchronizing point;
\item $(X_1\times X_2, \varphi_1 \times \varphi_2)$ is not finitely generated.
\end{enumerate}
The first item is standard so we omit the details. For the second, it is not difficult to see that if $x_1$ is synchronizing for $(X_1, \varphi_1)$ and $x_2$ is synchronizing for $(X_2, \varphi_2)$, then $(x_1, x_2)$ is synchronizing for $(X_1\times X_2, \varphi_1 \times \varphi_2)$.

For the third item, suppose that $(X_1\times X_2, \varphi_1 \times \varphi_2)$ is finitely generated. Then, by \cite[Theorem 2]{fried1987}, there is a shift of finite type $(\Sigma, \sigma)$ and a factor map $\pi : (\Sigma, \sigma) \rightarrow (X_1\times X_2, \varphi_1 \times \varphi_2)$. Let $\pi_1: X_1\times X_2 \rightarrow X_1$ be the projection map. It  is a factor map and the composition $\pi_1 \circ \pi$ is a factor map from a shift of finite type to $(X_1, \varphi_1)$. Using \cite[Theorem 2]{fried1987}, this contradicts the assumption that $(X_1, \varphi_1)$ is not finitely presented.
\end{proof}
\begin{example}
Let $(\Sigma, \sigma)$ be a synchronizing shift that is not sofic (e.g., the context-free shift or the shift discussed in Example \ref{nonSoficShift}) and $(\R^n/\Z^n, \varphi_2)$ be a mixing hyperbolic toral automorphism on the $n$-dimensional torus. Then, by Proposition \ref{productProp}, $(\Sigma \times \R^n/\Z^n, \sigma \times \varphi_2)$ is not finitely presented, but is synchronizing. Moreover, the underlying space of this dynamical system has dimension $n$.
\end{example}
The previous example shows that there are higher dimensional systems that are synchronizing, but not finitely presented. There are in fact many of these systems. One way to construct them is through skew product solenoids. This class of examples will be studied in detail in future work. For now, we include one explicit example.
\begin{example}
Let $(\Sigma, \sigma)$ be the shift space from Example \ref{nonSoficShift} and $(\Sigma^+, \sigma^+)$ be the associated one-sided shift space. Also let $K$ be the Klein bottle, $h_1: K \rightarrow K$ be the nine fold self covering map given in Figure \ref{dia9Klein}, and $h_2 : K \rightarrow K$ be the six fold self covering map defined in a similar way. 
\begin{figure}[h]
\begin{tikzpicture}[scale=0.75]
\draw[->-] (0,0)--(0,1);
\draw[->-] (1,0)--(1,1);
\draw[->>-] (1,0)--(0,0);
\draw[->>-] (0,1)--(1,1);

\node at (2,1/2) {$\mapsto$};

\draw[->-] (3,0)--(3,1); \draw[->-] (3,1)--(3,2); \draw[->-] (3,2)--(3,3);
\draw[->-] (6,0)--(6,1); \draw[->-] (6,1)--(6,2); \draw[->-] (6,2)--(6,3);
\draw[->>-] (3,3)--(4,3); \draw[->>-] (4,3)--(5,3); \draw[->>-] (5,3)--(6,3);
\draw[->>-] (6,0)--(5,0); \draw[->>-] (5,0)--(4,0); \draw[->>-] (4,0)--(3,0);

\draw[->-] (4,0)--(4,1); \draw[->-] (4,1)--(4,2); \draw[->-] (4,2)--(4,3);
\draw[->-] (5,0)--(5,1); \draw[->-] (5,1)--(5,2); \draw[->-] (5,2)--(5,3);

\draw[->>-] (3,1)--(4,1); \draw[->>-] (4,1)--(5,1); \draw[->>-] (5,1)--(6,1);
\draw[->>-] (6,2)--(5,2); \draw[->>-] (5,2)--(4,2); \draw[->>-] (4,2)--(3,2);
\end{tikzpicture}
\caption{A nine fold self-cover of the Klein bottle}
\label{dia9Klein}
\end{figure}
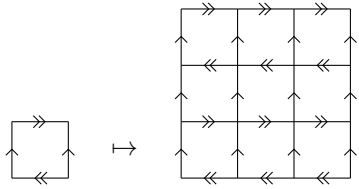

Given $(d_n)_{n \in \N}={\bold d} \in \Sigma^+$, we let 
\[ g_{{\bold d}} = \left\{ \begin{array}{cc} h_1 & d_0=\word{a} \\ h_2 & d_0\neq \word{a} \end{array} \right. \]
Based on these inputs, we define
\[ X=\{ ( ({\bold d}_0, y_0), ({\bold d}_1, y_1), ({\bold d}_2, y_2), \ldots ) \in \prod \Sigma^+ \times K \mid (\sigma^+({\bold d}_{i+1}), g_{{\bold d}_{i+1}}(y_{i+1}))=({\bold d}_i, y_i) \} \]
and 
\[ \varphi (({\bold d}_0, y_0), ({\bold d}_1, y_1), ({\bold d}_2, y_2), \ldots )= ((\sigma^+({\bold d}_0), g_{{\bold d}_0}(y_0)), ({\bold d}_0, y_0), ({\bold d}_1, y_1), \ldots ). \]
One can show that the dynamical system $(X, \varphi)$ is synchronizing, but not finitely presented. We will not discuss the details of the proof as this example fits within a general framework that will be studied in future work.
\end{example}

\section{Groupoids} \label{SecGroupoid}

We introduce several equivalence relations which capture notions of asymptotic equivalence of elements in an expansive dynamical system $(X,\varphi)$.  The definition of local conjugacy is due to Thomsen \cite{thomsen2010c}, including the stable and unstable versions --- although we reformulate Thomsen's original definition.  We will show that in the case of a synchronizing system much more can be said about these equivalence relations.

\subsection{The Local Conjugacy Relation}

Let $(X,\varphi)$ be an expansive dynamical system with $X$ a compact metric space and $\varphi : X \to X$ a homeomorphism.
\begin{definition}
We say that two points $x,y \in X$ are \emph{locally conjugate}, denoted $x \sim_\lc y$, if there exist two open neighborhoods $U$ and $V$ of $x$ and $y$ respectively, and a homeomorphism $\gamma : U \to V$ such that $\gamma(x) = y$ and \[ \lim_{n \to \pm \infty} \sup_{z \in U} \, d(\varphi^n(z), \varphi^n(\gamma(z))) = 0 \,. \]
We will denote the equivalence class of $x$ under the local conjugacy relation as $X^\lc(x)$, and we will also call the triple $(U,V,\gamma)$ a \emph{local conjugacy} from $x$ to $y$.
\end{definition}


One can show that local conjugacy is an equivalence relation. A crucial fact about local conjugacy is that, given two points that are locally conjugate, then a local conjugacy between the two points is essentially unique.  This is stated precisely in the following result of Thomsen \cite[Lemma 1.4]{thomsen2010c}.

\begin{lemma}\label{lemma:lc-unique}
Let $(U,V,\gamma)$ and $(U',V',\gamma')$ be two local conjugacies from $x$ to $y$.  Then there exists a local conjugacy $(U_0,V_0,\gamma_0)$ from $x$ to $y$ such that $x \in \displaystyle U_0 \subseteq U \cap U'$ and $\displaystyle \gamma|_{U_0} = \gamma'|_{U_0} = \gamma_0$.
\end{lemma}
It is clear that $x \sim_\lc y$ implies $x \sim_h y$ in any dynamical system. The converse is not always true, see \cite[Remark 1.13]{thomsen2010c} (or \cite[Example 4.4.1]{AndrewPhDThesis}) where the situation for the even shift is discussed. 

The situation is different for Smale spaces.  In fact for Smale spaces it is the case that two points are locally conjugate if and only if they are homoclinic.  In the following lemma we construct an explicit local conjugacy between homoclinic points using a construction involving the bracket map from Ruelle \cite{Ruelle1988}. This result is well-known, but its proof provides a nice example of the construction of a local conjugacy, so we have included it.
\begin{lemma}{\cite{Ruelle1988}}\label{lem:smale-h-implies-lc}
If $(X,\varphi)$ is a Smale space and $x \sim_h y$, then $x \sim_\lc y$.
\end{lemma}
\begin{proof}
We give a sketch of the proof given in \cite{putnam_1996}.  Fix $N \geq 0$ so that $d(\varphi^n(x), \varphi^n(y)) < \varepsilon_X$ for all $|n| \geq N$, which we know we can do because $x$ and $y$ are homoclinic points.  Next choose $\varepsilon > 0$ small enough that $(\lambda_X)^{-N} \varepsilon < \varepsilon_X$, which would mean that
\begin{align*}
    \varphi^n\left(X^\text{s}(x,\varepsilon)\right) &\subseteq X^\text{s}\left(\varphi^n(x), \varepsilon_X\right) \\
    \varphi^n\left(X^\text{u}(x,\varepsilon)\right) &\subseteq X^\text{u}\left(\varphi^n(x), \varepsilon_X\right)
\end{align*}
for all $n \in \{-N, \ldots N\}$.  Finally let $\delta > 0$ be small enough that $d(x,z) < \delta$ implies $[x,z] \in X^\text{s}(x,\varepsilon)$ and $[z,x] \in X^\text{u}(x,\varepsilon)$.  We can now define a map $\gamma$ for $z \in B_\delta(x)$ where $\gamma(z)$ is defined as \[ \gamma(z) = \left[ \varphi^{-N}\left[\varphi^N[z,x], \varphi^N(y)\right], \varphi^N\left[\varphi^{-N}(y), \varphi^{-N}[x,z]\right] \right] \,. \]  Note that by our specification $\delta$ and $\varepsilon$ this map is well-defined, and furthermore it is a homeomorphism onto its image.  The triple $\big( B_\delta(x), \gamma\left(B_\delta(x)\right), \gamma \big)$ is a local conjugacy from $x$ to $y$.
\end{proof}

It is worth commenting that the definition of local conjugacy is essentially informed by this construction of Ruelle \cite{Ruelle1988}, and the fact that the local conjugacy and homoclinic relations agree for Smale spaces is what makes local conjugacy a suitable relation for generalizing Ruelle's and Putnam's $C^\ast$-algebraic constructions for Smale spaces to the class of all expansive dynamical systems.

Next we show that local conjugacy is invariant under $\varphi$ for any expansive system.

\begin{lemma}\label{lem:lc-invariant-under-map}
Let $(X,\varphi)$ be an expansive system and $x,y \in X$.  Then $x \sim_\lc y$ if and only if $\varphi(x) \sim_\lc \varphi(y)$.
\end{lemma}
\begin{proof}
Let $(U,V,\gamma)$ be a local conjugacy from $x$ to $y$, then we will show that $(\varphi(U), \varphi(V), \varphi \circ \gamma \circ \varphi^{-1})$ is a local conjugacy from $\varphi(x)$ to $\varphi(y)$.  We compute
\begin{align*}
    &\lim_{n\to\pm\infty} \sup_{z' \in \varphi(U)} d\left(\varphi^n(z'), \varphi^n(\varphi \circ \gamma \circ \varphi^{-1}(z'))\right) \\
    &= \lim_{n\to\pm\infty} \sup_{\substack{z \in U \\ z' = \varphi(z)}} d\left(\varphi^{n+1}(z), \varphi^{n+1}(\gamma(z))\right) \\
    &= 0 \,.
\end{align*}
A similar computation shows that $x \sim_\lc y$ if $\varphi(x) \sim_\lc \varphi(y)$.
\end{proof}

Next, we will need the useful fact that a local conjugacy preserves the local stable and unstable directions of points in a uniform sense.  Recall we are always assuming the metric on an expansive system is the adapted metric.

\begin{lemma}\label{lem:lc-uniform}
Let $(X,\varphi)$ be an expansive dynamical system.  If $(U,V,\gamma)$ is a local conjugacy and $W \subseteq U$ is compact, then for all $\varepsilon > 0$ there is a $\delta > 0$ such that
\begin{align*}
    \gamma\left(X^\text{s}(z, \delta)\right) &\subseteq X^\text{s}(\gamma(z), \varepsilon) \\
    \gamma\left(X^\text{u}(z, \delta)\right) &\subseteq X^\text{u}(\gamma(z), \varepsilon)
\end{align*}
for all $z \in W$.
\end{lemma}
\begin{proof}
Since $(U,V,\gamma)$ is a local conjugacy, we can find $0 < \delta_1 \leq \frac{\varepsilon}{3}$ such that for all $z \in W$ we have $X^\text{s}(z, \delta_1) \subseteq U$ and $X^\text{u}(z, \delta_1) \subseteq U$, and if $z' \in X^\text{s}(z, \delta_1)$ then there exists $N$ such that
\[ d(\varphi^n(\gamma(z)), \varphi^n(\gamma(z'))) \leq d(\varphi^n(\gamma(z)), \varphi^n(z)) + d(\varphi^n(z), \varphi^n(z')) + d(\varphi^n(z'), \varphi^n(\gamma(z'))) \leq \varepsilon \]
for all $n>N$.  In other words there exists $\delta_1$ such that $\gamma(z') \in \varphi^{-N}\left(X^\text{s}\left(\varphi^N(\gamma(z)), \varepsilon\right)\right)$ for all $z \in W$ and $z' \in X^\text{s}(z, \delta_1)$.  Additionally, by the uniform continuity of $\varphi$ and $\gamma$ restricted to $W$, we may find $\delta_2 > 0$ such that for all $z,z' \in W$ we have $d(z,z') < \delta_2$ implies $d(\varphi^n(\gamma(z)), \varphi^n(\gamma(z'))) < \varepsilon$ for all $0\le n <N$.

Letting $\delta = \min\{\delta_1, \delta_2\}$, for all $z \in W$ we then have that $z' \in X^\text{s}(z, \delta)$ implies $\gamma(z') \in X^\text{s}(\gamma(z), \varepsilon)$, or in other words there is a $\delta > 0$ such that for every $z \in W$ \[ \gamma\left(X^\text{s}(z, \delta)\right) \subseteq X^\text{s}(\gamma(z), \varepsilon) \,. \]
The proof for the local unstable set of $z \in W$ follows similarly.
\end{proof}

We now prove the important fact that the property of being synchronizing is invariant under the local conjugacy relation.  In particular this will show that if $x$ is synchronizing then then entire equivalence class of $x$ under local conjugacy is also synchronizing.  Hence the synchronizing points determine a sub-equivalence relation.  Note that the following is not true for the homoclinic relation, that is if $x$ is synchronizing and $x \sim_\text{h} y$ then it does not follow that $y$ is synchronizing. An example can be constructed for the even shift, the details are similar to \cite[Remark 1.13]{thomsen2010c}.

\begin{prop}\label{prop:lc-preserves-sync}
Let $(X,\varphi)$ be an expansive system and $x,y \in X$.  If $x$ is synchronizing and $x \sim_\lc y$, then $y$ is synchronizing.
\end{prop}
\begin{proof}
Let $(U,V,\gamma)$ be a local conjugacy from $x$ to $y$.  Since $x$ is synchronizing we know that there exists $\delta_x > 0$ such that $X^\text{u}(x, \delta_x) \times X^\text{s}(x, \delta_x) \subseteq D_\varepsilon$ and $[-,-]$ is a homeomorphism from this set onto its image $R$ --- which we can assume is contained in $U$.  Since $R$ is a compact subset of $U$, there is a $\delta > 0$ such that Lemma \ref{lem:lc-uniform} holds for $z \in R$.  Let $\delta'_x = \min\{\delta_x, \delta\}$, so we have that, with $R' \subseteq R$ being the homeomorphic image of $X^\text{u}(x, \delta'_x) \times X^\text{s}(x, \delta'_x)$ under $[-,-]$,
\begin{align*}
    \gamma\left(X^\text{s}(z, \delta'_x)\right) &\subseteq X^\text{s}(\gamma(z), \varepsilon) \\
    \gamma\left(X^\text{u}(z, \delta'_x)\right) &\subseteq X^\text{u}(\gamma(z), \varepsilon)
\end{align*}
for all $z \in R'$.

Next we will show that $y$ is indeed synchronizing in the sense that we can define $[z,z']$ for $z,z'$ close enough to $y$.  Let $\delta_y > 0$ be such that $X^\text{u}(y, \delta_y) \times X^\text{s}(y, \delta_y) \subseteq V \times V$ and if $(z,z') \in X^\text{u}(y, \delta_y) \times X^\text{s}(y, \delta_y)$ then $(\gamma^{-1}(z), \gamma^{-1}(z')) \in X^\text{u}(x, \delta'_x) \times X^\text{s}(x, \delta'_x)$.  Hence $[\gamma^{-1}(z), \gamma^{-1}(z')]$ is well-defined.  Since $[\gamma^{-1}(z), \gamma^{-1}(z')] \in R' \subseteq U$, we can define $\gamma\left([\gamma^{-1}(z), \gamma^{-1}(z')]\right) \in V$.  We will show that $[z,z']$ is defined and that \[ [z,z'] = \gamma\left([\gamma^{-1}(z), \gamma^{-1}(z')]\right) \,. \]  From above we have that 
\begin{align*}
    \gamma\left([\gamma^{-1}(z), \gamma^{-1}(z')]\right) &\in \gamma\left( X^\text{s}(\gamma^{-1}(z), \delta'_x) \cap X^\text{u}(\gamma^{-1}(z'), \delta'_x) \right) \\
    &= \gamma\left( X^\text{s}(\gamma^{-1}(z), \delta'_x)\right) \cap \gamma\left(X^\text{u}(\gamma^{-1}(z'), \delta'_x)\right) \\
    &\subseteq  X^\text{s}(z, \varepsilon) \cap X^\text{u}(z', \varepsilon)
\end{align*}
meaning $[z,z']$ is defined as just mentioned.  Note that $\gamma^{-1}(z)$ and $\gamma^{-1}(z')$ are both in $R'$, so the last step is justified.  Hence $X^\text{u}(y, \delta_y) \times X^\text{s}(y, \delta_y) \subseteq D_\varepsilon$ and this shows $y$ is synchronizing.
\end{proof}

\subsection{The Stable/Unstable Local Conjugacy Relations}

We also have the notions of local conjugacy which only hold in the limit in one direction.  We remark that, similar to local conjugacy, these respectively imply stable and unstable equivalence.  However, the converse is not true; again, an example can be constructed for the even shift and the details are similar to \cite[Remark 1.13]{thomsen2010c}.

\begin{definition}
Let $(X,\varphi)$ be an expansive dynamical system with $x,y \in X$.  Suppose $\delta,\delta' > 0$ and $\gamma : X^\text{u}(x,\delta) \to X^\text{u}(y,\delta')$ is a homeomorphism onto its image such that $\gamma(x) = y$ and \[ \lim_{n \to \infty} \sup_{z \in U} d(\varphi^n(z), \varphi^n(\gamma(z))) = 0 \,. \]  Then $\gamma$ is a called a \emph{stable local conjugacy} and we say that $x \sim_\lcs y$.  We will denote the stable local conjugacy equivalence class of $x$ as $X^\lcs(x)$.
\end{definition}

\begin{definition}
Let $(X,\varphi)$ be an expansive dynamical system with $x,y \in X$.  Suppose $\delta,\delta' > 0$ and $\gamma : X^\text{s}(x,\delta) \to X^\text{s}(y,\delta')$ is a homeomorphism onto its image such that $\gamma(x) = y$ and \[ \lim_{n \to \infty} \sup_{z \in U} d(\varphi^{-n}(z), \varphi^{-n}(\gamma(z))) = 0 \,. \]  Then $\gamma$ is a called an \emph{unstable local conjugacy} and we say that $x \sim_\lcu y$.  We will denote the ustable local conjugacy equivalence class of $x$ as $X^\lcu(x)$.
\end{definition}

Note that these are indeed equivalence relations in the same way that local conjugacy is an equivalence relation.  Also note that if $x \sim_\lc y$ via a local conjugacy $\gamma$, we get both a stable and unstable local conjugacy from $x$ to $y$ by making the restrictions $\gamma|_{X^\text{u}(x,\delta)}$ and $\gamma|_{X^\text{s}(x,\delta)}$ respectively.  In short, $x \sim_\lc y$ implies $x \sim_\lcs y$ and $x \sim_\lcu y$.  Lastly note that both stable and unstable local conjugacy is unique in the same sense as Lemma \ref{lemma:lc-unique}, see \cite[Lemma 1.4]{thomsen2010c}.

We have the following lemma, the proof of which is the same as the proof of Lemma \ref{lem:lc-invariant-under-map}.

\begin{lemma}\label{lem:lcs-lcu-invariant-under-map}
Let $(X,\varphi)$ be an expansive system and $x,y \in X$.  Then
\begin{enumerate}[(i)]
    \item $x \sim_\lcs y$ if and only if $\varphi(x) \sim_\lcs \varphi(y)$, and
    \item $x \sim_\lcu y$ if and only if $\varphi(x) \sim_\lcu \varphi(y)$.
\end{enumerate}
\end{lemma}

The following proposition gives conditions under which $x \sim_\lcu y$ and $x \sim_\lcs y$ implies $x \sim_\lc y$, namely that $x$ and $y$ must both be synchronizing.  This is a generalization of the fact that $x \sim_\text{u} y$ and $x \sim_\text{s} y$ implies $x \sim_\text{h} y$ in a Smale space, and is more evidence that synchronizing points are points that behave like points in a Smale space.

\begin{prop}\label{prop:sync-lcu-lcs-imply-lc}
Let $(X,\varphi)$ be a synchronizing system with $x,y \in X_\sync$.  If $x \sim_\lcs y$ and $x \sim_\lcu y$, then $x \sim_\lc y$.
\end{prop}
\begin{proof}
Assume that $\varepsilon_0$ is fixed as in Lemma \ref{lem:epsilon-naught-existence}.  Let $\lambda$ be the constant associated to the adapted metric, see Section \ref{sec:adapted-metric}.  Since $x$ and $y$ are both synchronizing, we have the rectangular neighborhoods
\begin{align*}
    R_x &\cong X^\text{u}(x,\delta_x) \times X^\text{s}(x, \delta_x) \text{ , and } \\
    R_y &\cong X^\text{u}(y,\delta_y) \times X^\text{s}(y, \delta_y)
\end{align*}
by Proposition \ref{rectangle-nbhd}.  Since $x \sim_\lcs y$ and $x \sim_\lcu y$ we have stable and unstable local conjugacies which we can assume are of the form
\begin{align*}
    \gamma^\text{u} &: X^\text{u}(x, \delta) \to X^\text{u}(y, \delta_y) \text{ , and } \\
    \gamma^\text{s} &: X^\text{s}(x, \delta) \to X^\text{s}(y, \delta_y)
\end{align*}
for some $0 < \delta \leq \delta_x$ small enough that $\gamma^\text{u}$ and $\gamma^\text{s}$ are both homeomorphisms onto their respective images.  Let $U$ denote the image of $X^\text{u}(x, \delta) \times X^\text{s}(x, \delta)$ under $[-,-]$, noting that $h_x$ (as defined in the proof of Proposition \ref{rectangle-nbhd}) restricted to $U$ is a homeomorphism since $\delta \leq \delta_x$.  We define $\gamma$ as \[ \gamma(z) = \left[\gamma^\text{u}\left([z, x]\right), \gamma^\text{s}\left([x, z]\right)\right] \] for $z \in U$, which is a homeomorphism since it is the composition of homeomorphisms $[-,-] \circ (\gamma^\text{u} \times \gamma^\text{s}) \circ (h_x)|_U$.  Let $V$ be the image of this map in $R_y$, then $\gamma : U \to V$ is the homeomorphism that we wish to show is a local conjugacy from $x$ to $y$.

We will only show the limit in the stable direction as the argument for the unstable direction is very similar.  Fix $\varepsilon > 0$, then observe that
\begin{align*}
    \lim_{n \to \infty} \sup_{z \in U} d(\varphi^n(z), \varphi^n(\gamma(z))) &= \lim_{n \to \infty} \sup_{z \in U} d(\varphi^n(z), \varphi^n(\left[\gamma^\text{u}\left([z, x]\right), \gamma^\text{s}\left([x, z]\right)\right])) \\
    &= \lim_{n \to \infty} \sup_{(z_1, z_2) \in X^\text{u}(x, \delta) \times X^\text{s}(x, \delta)} d(\varphi^n([z_1, z_2]), \varphi^n(\left[\gamma^\text{u}\left(z_1\right), \gamma^\text{s}\left(z_2\right)\right])) \,.
\end{align*}
Since $\gamma^\text{u}$ is a stable local conjugacy, we can find $N$ large enough that $d(\varphi^n(z_1), \varphi^n(\gamma^\text{u}(z_1))) < \frac{\varepsilon}{3}$ for all $n \geq N$ and $z_1 \in X^\text{u}(x, \delta)$.  Since we are assuming $d$ is the adapted metric, we will also increase $N$ if necessary to ensure that $\lambda^n \diam(X) < \frac{\varepsilon}{3}$ for all $n \geq N$.  Then
\begin{align*}
    d(\varphi^n([z_1, z_2]), \varphi^n(\left[\gamma^\text{u}\left(z_1\right), \gamma^\text{s}\left(z_2\right)\right])) &\leq d(\varphi^n([z_1, z_2]), \varphi^n(z_1)) + d(\varphi^n(z_1), \varphi^n(\gamma^\text{u}(z_1))) \\ & \hspace{3.5cm} + d(\varphi^n(\gamma^\text{u}(z_1))), \varphi^n(\left[\gamma^\text{u}\left(z_1\right), \gamma^\text{s}\left(z_2\right)\right])) \\
    &\leq \lambda^n d([z_1, z_2], z_1) + d(\varphi^n(z_1), \varphi^n(\gamma^\text{u}(z_1))) \\ & \hspace{3.5cm} + \lambda^n d(\gamma^\text{u}(z_1), \left[\gamma^\text{u}\left(z_1\right), \gamma^\text{s}\left(z_2\right)\right]) \\
    &\leq d(\varphi^n(z_1), \varphi^n(\gamma^\text{u}(z_1))) + 2\lambda^n \diam(X) \\
    &< \varepsilon
\end{align*}
for all $n \geq N$.  Hence \[ \lim_{n \to \infty} \sup_{z \in U} d(\varphi^n(z), \varphi^n(\gamma(z))) = 0 \,. \]  The proof that the limit in the unstable direction vanishes is similar, and thus $\gamma : U \to V$ is a local conjugacy from $x$ to $y$, that is, $x \sim_\lc y$.
\end{proof}

\begin{lemma}\label{lem:rect-nbhd-implies-lcu-lcs}
Let $(X,\varphi)$ be a synchronizing system with $x \in X$ synchronizing.  Then if $R$ is a product neighborhood of $x$ and $y \in \interior(R)$, we have
\begin{enumerate}[(i)]
    \item $y \sim_\lcu [x,y] \sim_\lcs x$, and
    \item $y \sim_\lcs [y,x] \sim_\lcu x$.
\end{enumerate}
\end{lemma}
\begin{proof}
We show the proof of (i) and the proof of (ii) is similar.  Assume $R \cong X^\text{u}(x, \delta_x) \times X^\text{s}(x, \delta_x)$.  Assume $\delta > 0$ is small enough that $X^\text{s}(y, \delta) \subseteq R$.  Then we define $\gamma : X^\text{s}(y, \delta) \to X^\text{s}(x, \delta_x)$ by \[ \gamma(z) = [x, z] \] for $z \in X^\text{s}(y, \delta)$.  Note that this is a homeomorphism onto its image and that $\gamma(y) = [x, y]$.  Next we compute
\begin{align*}
    \lim_{n \to \infty} \sup_{z \in X^\text{s}(y,\delta)} d(\varphi^{-n}(z), \varphi^{-n}(\gamma(z))) &= \lim_{n \to \infty} \sup_{z \in X^\text{s}(y,\delta)} d(\varphi^{-n}(z), \varphi^{-n}([x, z])) \\
    &\leq \lim_{n \to \infty} \sup_{z \in X^\text{s}(y,\delta)} \lambda^n d(z, [x, z]) \\
    &\leq \lim_{n \to \infty} \lambda^n \diam(R) \\
    &= 0 \,.
\end{align*}
Hence $\gamma$ is an unstable local conjugacy, so $y \sim_\lcu [x,y]$.  A very similar argument shows that $\gamma'(z) = [z, x]$ for $z \in X^\text{u}([x,y], \delta')$, and for $\delta' > 0$ small enough, is a stable local conjugacy from $[x,y]$ to $x$.
\end{proof}

\subsection{Local Conjugacy and Synchronizing Periodic Points}

We will be interested in local conjugacy with respect to periodic points that are also synchronizing.  By Proposition \ref{theorem:sync-points-dense}, we know that there are a lot of these.  In fact, since $Per(X,\varphi)$ is dense in $X$ and $X_\sync$ is an open dense subset of $X$, we can conclude that synchronizing periodic points are also dense in $X$.  The synchronizing periodic points will behave similar to periodic points in a Smale space, and this is what will make them suitable for proving results about synchronizing spaces that extend results of Putnam \cite{putnam_1996} and Putnam--Spielberg \cite{putnam99} about Smale spaces.

\begin{lemma}
Let $(X,\varphi)$ be an expansive system, $x \in X$, and $p \in X$ a synchronizing periodic point.  If $x \sim_\text{s} p$, $x \sim_\text{u} p$, $x \sim_\text{h} p$, or $x \sim_\lc p$, then $x$ is also synchronizing.
\end{lemma}
\begin{proof}
Let $r$ denote the period of $p$.  We will prove this for the case when $x \sim_\text{s} p$ as the proofs for the other cases either follow directly or are similar.  By Definition \ref{def:synchronizing-point}, we know that $(p,p) \in \text{int}(D_\varepsilon)$ for some $\varepsilon > 0$.  Since $x \sim_\text{s} p$, we also know that for some $N$ large enough $(\varphi^{rn}(x), \varphi^{rn}(x)) \in \text{int}(D_\varepsilon)$ for all $n \geq N$.  This means in particular that $\varphi^{rN}(x)$ is synchronizing, and by Lemma \ref{lemma:sync-invariant-under-map} we conclude $x$ must be synchronizing as well.
\end{proof}

We will now show that the stable and stable local conjugacy relations agree for a synchronizing periodic point $p$.  This is not necessarily true for an arbitrary synchronizing point. An example of such a point can be found in the even shift. We leave the details to the reader.

\begin{lemma}\label{lem:sync-periodic-lc-is-same}
Let $(X,\varphi)$ be an expansive system and $p \in X$ a synchronizing periodic point.  Then $X^\lcu(p) = X^\text{u}(p)$ and $X^\lcs(p) = X^\text{s}(p)$.  In other words, $p \sim_\text{s} x$ if and only if $p \sim_\lcs x$ and $p \sim_\text{u} x$ if and only if $p \sim_\lcu x$ for any $x \in X$.
\end{lemma}
\begin{proof}
We will show the proof that $p \sim_\text{s} x$ implies $p \sim_\lcs x$, and the result that $p \sim_\text{u} x$ implies $p \sim_\lcu x$ is very similar.  By definition we have $X^\lcs(p) \subseteq X^\text{s}(p)$, so we need to show $X^\text{s}(p) \subseteq X^\lcs(p)$.  Let $R = X^\text{u}(p, \delta_p) \times X^\text{s}(p, \delta_p)$ be a product neighborhood of $p$, and let $x \in X^\text{s}(p)$.  Then, there is an $N$ such that $\varphi^{rn}(x) \in R$ for all $n \geq N$ where $r$ is the period of $p$.  In particular, $\varphi^{rn}(x) \in X^\text{s}(p, \delta_p)$ for all $n \geq N$.  Furthermore note that $\varphi^{rN} \sim_\lcs p$ by Lemma \ref{lem:rect-nbhd-implies-lcu-lcs}, and so by Lemma \ref{lem:lcs-lcu-invariant-under-map} we have $x \sim_\lcs \varphi^{-rN}(p) = p$.
\end{proof}

\begin{corollary} 
Let $(X,\varphi)$ be an expansive system, $x \in X$, and $p \in X$ a synchronizing point which is also periodic with period $n$.  Then $X^\text{h}(p) = X^\lc(p)$.
\end{corollary}
\begin{proof}
This follows from Lemma \ref{lem:sync-periodic-lc-is-same} and Proposition \ref{prop:sync-lcu-lcs-imply-lc}.
\end{proof}


For the next several results, note that it is necessary to assume that a given synchronizing system is mixing in order to achieve the density of $X^\text{u}(p)$ and $X^\text{s}(p)$. In the first of these, we extend a property of Smale spaces to the case of synchronizing dynamical systems in the specific case of a synchronizing periodic point, see \cite{ruelle_2004}. We will start with two more general results. We learned of the proof technique used in these lemmas from Putnam \cite[Theorem 3.1.4]{putnamNotes19}.

\begin{lemma}
Let $(X,\varphi)$ be a mixing synchronizing system. Suppose that $x\in X$ and there exists $x_0$ in the closure of the forward orbit of $x$ that is synchronizing. Then $X^\text{s}(x)$ is dense in $X$. 
\end{lemma}
\begin{proof}
Let $y\in X$ and $\delta>0$. We will show that $X^\text{s}(x) \cap B_{\delta}(y)$ is non-empty. Let $U=B_{\frac{\delta}{2}}(y)$ and $\epsilon^{\prime}=\min\{ \frac{\delta}{2}, \eta\}$ where $\eta$ is the constant associated to the adapted metric. Using the fact that $x_0$ synchronizing and Lemma \ref{lemma:bracket-epsilon-continuity}, there exists $\delta^{\prime}$ such that
\begin{enumerate}
\item $B_{\delta^{\prime}}(x_0) \subseteq R$ where $R$ is a rectangular neighborhood of $x_0$ and
\item $[x',y'] \in X^\text{s}(x', \varepsilon') \cap X^\text{u}(y', \varepsilon')$ whenever $d(x_0, x') < \delta'$ and $d(x_0, y') < \delta'$.
\end{enumerate}
Let $V=B_{\delta^{\prime}}(x_0)$ and apply the definition of mixing to $U$ to get $N\in \N$ such that
\[
\varphi^n(U) \cap V \neq \emptyset \hbox{ for each }n\ge N.
\]
Since $x_0$ is in the closure of the forward orbit of $x$, there exists $n \ge N$ such that $\varphi^{n}(x) \in B_{\delta^{\prime}}(x_0)$. Using the fact that $\varphi^{n}(U) \cap V$ is non-empty, there exists $w \in U=B_{\frac{\delta}{2}}(y)$ such that $\varphi^{n}(w) \in V=B_{\delta^{\prime}}(x_0)$. 

In particular, the bracket $[\varphi^n(x), \varphi^n(w)]$ is well defined so we can form
\[
z=\varphi^{-n}( [\varphi^n(x), \varphi^n(w)] )
\]
By the second property of $\delta^{\prime}$ above, we have that $\varphi^{n}(z) \in X^u(\varphi^{n}(w), \epsilon^{\prime})$. Using properties of the adapted metric, we have that
\[
d(w, z) \le \lambda^n \epsilon^{\prime} < \epsilon^{\prime}.
\]
Hence
\[
d(y, z) \le d(y, w) + d(w, z) < \frac{\delta}{2} + \epsilon^{\prime}< \delta
\]
and we have that $z \in B_{\delta}(y)$. Notice also that $z\in X^s(x)$ since $\varphi^n(z) \in X^s( \varphi^n(x), \epsilon^{\prime})$, so $X^\text{s}(x) \cap B_{\delta}(y)$ is non-empty as required.
\end{proof}
\begin{lemma}
Let $(X,\varphi)$ be a mixing synchronizing system. Suppose that $x\in X$ and there exists $x_0$ in the closure of the backward orbit of $x$ that is synchronizing. Then $X^\text{u}(x)$ is dense in $X$. 
\end{lemma}
\begin{proof}
This follows from the previous result by replacing $\varphi$ with $\varphi^{-1}$.
\end{proof}

\begin{lemma}\label{lem:sync-per-stable-unstable-dense}
Let $(X,\varphi)$ be a mixing synchronizing system.  If $p \in X$ is a synchronizing periodic point, then $X^\text{s}(p)$ and $X^\text{u}(p)$ are dense as subsets of $X$.  If $q \in X$ is another synchronizing periodic point, then $X^\text{s}(p) \cap X^\text{u}(q)$ is also dense in $X$.
\end{lemma}
\begin{proof}
The first part follows from the previous two lemmas because $p$ is by assumption synchronizing and is in the closure of the forward orbit and also the closure of the backward orbit because it is periodic. 

For the second part, that is to show that $X^\text{s}(p) \cap X^\text{u}(q)$ is also dense in $X$, first let $z$ be any synchronizing point with rectangular neighborhood $R$.  Note that we can make $\diam(R)$ arbitrarily small by the continuity of $[-,-]$.  Then by the density of $X^\text{s}(p)$ and $X^\text{u}(q)$, let $x, y \in R$ be such that $x \in X^\text{s}(p)$ and $y \in X^\text{u}(q)$.  Then let $w = [x, y]$, which is in $X^\text{s}(p) \cap X^\text{u}(q)$.  Since synchronizing points are dense in $X$ and $R$ is an arbitrarily small rectangular neighborhood, it follows that $X^\text{s}(p) \cap X^\text{u}(q)$ is dense in $X$.
\end{proof}

It is important to point out that Lemma \ref{lem:sync-per-stable-unstable-dense} says $X^\text{u}(p)$ and $X^\text{s}(p)$ are dense as subsets of $X$, despite the fact that we have topologized them as described in Section \ref{sec:global-stable-unstable-sets}.  In fact the density of $X^\text{u}(p)$ and $X^\text{s}(p)$ is precisely why we do not use the subspace topology. The next result will used later in the proof of Theorem \ref{th:sync-morita-equivalence}.

\begin{lemma}\label{lem:sync-to-lcs-lcu-existence}
Let $(X,\varphi)$ be a mixing synchronizing system with $p,q \in X$ both synchronizing periodic points.  If $z$ is a synchronizing point, then there exists $x \in X^\text{u}(p)$ and $y \in X^\text{s}(q)$ satisfying \[ x \sim_\lcs z \sim_\lcu y \,. \]
\end{lemma}
\begin{proof}
By Lemma \ref{lem:sync-per-stable-unstable-dense}, $X^\text{u}(p)$ and $X^\text{s}(q)$ are both dense in $X$.  Let $R \cong X^\text{u}(z, \delta_z) \times X^\text{s}(z, \delta_z)$ denote a rectangular neighborhood of $z$.  Then there are $x', y' \in \text{int}(R)$ such that $x' \in X^\text{u}(p)$ and $y' \in X^\text{s}(q)$.  Let $x = [z, x']$ and $y = [y', z]$, then $x \in X^\text{u}(p)$ and $y \in X^\text{s}(q)$.  The result now follows from Lemma \ref{lem:rect-nbhd-implies-lcu-lcs}.
\end{proof}

For the next result we must assume some conditions on the density of the stable and unstable local conjugacy equivalence classes of elements of $X^\text{u}(p)$ and $X^\text{s}(p)$.  Note that these do not hold for all mixing synchronizing systems. However, they do hold for mixing finitely presented systems as we will show in Lemma \ref{lem:fp-lcu-lcs-dense}.  It is important to note that in Lemma \ref{lem:lcs-lcu-to-sync-existence} that we are considering $X^\text{u}(p)$ and $X^\text{s}(p)$ with their respective locally compact Hausdorff topologies as described in Section \ref{sec:global-stable-unstable-sets}.  Hence, for example, $X^\lcs(x) \cap X^\text{u}(p)$ being dense in $X^\text{u}(p)$ is stronger than being dense in $X$.

\begin{lemma}\label{lem:lcs-lcu-to-sync-existence}
Let $(X,\varphi)$ be a mixing synchronizing system with $p,q \in X$ both synchronizing periodic points.  Furthermore assume that, for all $x \in X^\text{u}(p)$ and $y \in X^\text{s}(q)$,
\begin{enumerate}[(i)]
    \item $X^\lcs(x) \cap X^\text{u}(p)$ is dense in $X^\text{u}(p)$, and
    \item $X^\lcu(y) \cap X^\text{s}(q)$ is dense in $X^\text{s}(q)$.
\end{enumerate}
Then if $x \in X^\text{u}(p)$ and $y \in X^\text{s}(q)$, then there exists a synchronizing point $z \in X_\sync$ satisfying \[ x \sim_\lcs z \sim_\lcu y \,. \]
\end{lemma}
\begin{proof}
Let $x \in X^\text{u}(p)$ and $y \in X^\text{s}(q)$.  Since $y$ is synchronizing, let $R$ be a product neighborhood of $y$.  Then since $X^\text{u}(p)$ is dense in $X$ (Lemma \ref{lem:sync-per-stable-unstable-dense}) and $X^\lcs(x)$ is dense $X^\text{u}(p)$ by assumption, there is an $x' \in R$ such that $x \sim_\lcs x'$.  Let $z = [x', y]$.  By Lemma \ref{lem:rect-nbhd-implies-lcu-lcs} we have \[ x \sim_\lcs x' \sim_\lcs z \sim_\lcu y \,. \]  Hence $x \sim_\lcs z$ and the result is proved.
\end{proof}

\section{$C^\ast$-Algebras from Expansive Dynamical Systems}\label{c-star-algebras}

In this section we construct $C^\ast$-algebras from the local conjugacy equivalence relations in the previous chapter.  This is following the work of Thomsen \cite{thomsen2010c} and Putnam \cite{putnam_1996, putnam99}.  Thomsen constructs the homoclinic algebra for any expansive dynamical system, and the heteroclinic algebras for expansive systems with dense periodic points \cite{thomsen2010c}.  These generalize constructions in the Smale space case \cite{putnam_1996}.  We will introduce these algebras and certain variants of them (see Remark \ref{rem:differentalgebras}) and then study their structure in the particular case of synchronizing systems.

\subsection{The Homoclinic Algebra}\label{sec:homoclinic-algebra}

We now want to construct a $C^\ast$-algebra from the local conjugacy relation for an expansive dynamical system.  We first define the groupoid \[ G^\lc(X, \varphi) = \left\{ (x,y) \in X \times X \mid x \sim_\lc y \right\} \] with groupoid composition given by $(x,y)(y',z) = (x,z)$ whenever $y = y'$.  However, instead of topologizing $G^\lc(X, \varphi)$ as a subspace of $X \times X$, we generate a topology from the sets $\{(z,\gamma(z)) \mid z \in U\}$ where $(U, V, \gamma)$ is a local conjugacy.  The fact that these sets are a base for a topology on $G^\lc(X, \varphi)$ follows from Lemma \ref{lemma:lc-unique}, also see \cite{thomsen2010c}.  With this topology, $G^\lc(X, \varphi)$ is indeed an \'{e}tale groupoid, see \cite[Theorem 1.7]{thomsen2010c}.

\begin{definition}
The \emph{homoclinic algebra} $A(X,\varphi)$ of an expansive dynamical system $(X,\varphi)$ is the reduced $C^\ast$-algebra $C_r^\ast(G^\lc(X, \varphi))$ in the sense of Renault \cite{renault1980groupoid}.
\end{definition}

If $(X,\varphi)$ is a Smale space, then by \ref{lem:smale-h-implies-lc},the $C^\ast$-algebra $A(X,\varphi)$ is exactly the asymptotic algebra defined by Putnam in \cite{putnam_1996}. It also is worth mentioning that the topology on $G^\lc(X,\varphi)$ is distinct from subspace topology obtained from $G^\lc(X, \varphi)\subseteq X\times X$ with the product topology.  



\subsection{Asymptotic Commutativity}

We prove a result that extends Putnam's result on the asymptotic commutativity of the homoclinic algebra for a Smale space \cite{putnam_1996}.  Note that this holds for any expansive dynamical system.  Some remarks are in order as to why this proof holds in wider generality.  Essentially this boils down to the fact that the local conjugacy relation incorporates local product structure, and is stronger than the homoclinic relation.  Note also that we do no require any recurrence properties, and they are indeed also not required in Putnam's proof of Theorem 3.2 in \cite{putnam_1996}.  However, a similar result involving the synchronizing heteroclinic algebras (see Section \ref{sec:heteroclinic-algebras}), one would likely need to require mixing as in the Smale space case.

There is a natural $*$-automorphism $\alpha$ of $A(X, \varphi)$ defined by \[ (\alpha(f))(x,y) = f(\varphi^{-1}(x), \varphi^{-1}(y)) \] for $f \in C_\text{c}(G^\lc(X,\varphi))$, see \cite{thomsen2010c}.  We wish to show that $(A(X, \varphi),\alpha)$ is asymptotically abelian in the following sense.

\begin{theorem}
Suppose $(X,\varphi)$ is an expansive dynamical system, and let $\alpha$ be the $*$-automorphism induced by $\varphi$.  Then for all $a,b \in A(X, \varphi)$, \[ \lim_{n\to\infty} \| \alpha^n(a)b - b\alpha^n(a) \| = 0 \,.\]  In other words, $(A(x,\varphi), \alpha)$ is \emph{asymptotically abelian}.
\end{theorem}
\begin{proof}
It suffices to prove this for $a,b \in C_\text{c}(G^\lc(X, \varphi))$, and furthermore we also assume
\begin{align*}
    \text{supp}(a) = \{ (z, \gamma_a(z)) \mid z \in K_a \} &\subseteq \{ (z, \gamma_a(z)) \mid z \in U_a \} \text{ and} \\
    \text{supp}(b) = \{ (z, \gamma_b(z)) \mid z \in K_b \}&\subseteq \{ (z, \gamma_b(z)) \mid z \in U_b \}
\end{align*}
where $K_a \subseteq U_a$ and $K_b \subseteq U_b$ are compact and $(U_a,V_a,\gamma_a)$ and $(U_b,V_b,\gamma_b)$ are both local conjugacies.

Let $\varepsilon_0 > 0$ be small enough that all $z \in X$ within $\varepsilon_0$ of $K_a$ are in $U_a$, and likewise for all $z \in X$ within $\varepsilon_0$ of $K_b$.  Then, since $\gamma_a$ and $\gamma_b$ are both local conjugacies, we can choose $N_0 \geq 0$ such that $d(\varphi^n(z), \varphi^n(\gamma_a(z))) < \varepsilon_0$ and $d(\varphi^n(z'), \varphi^n(\gamma_b(z'))) < \varepsilon_0$ for all $z \in U_a$, $z' \in U_b$, and $|n| \geq N_0$.  Lastly define $K_a' = \{ z \in X \mid \exists z' \in K_a \text{ s.t. } d(z,z') \leq \varepsilon_0 \}$ and $K_b'$ similarly.  Thus we have $K_a \subseteq K_a' \subseteq U_a$, $K_b \subseteq K_b' \subseteq U_b$ and both $K_a'$ and $K_b'$ are compact.

Now by our assumptions of the support of $a$ and $b$, we must have, for any $n$,
\begin{align*}
    (\alpha^n(a)b)(x,y) &= \sum_{(x,z) \in G^\lc(X, \varphi)} a(\varphi^{-n}(x),\varphi^{-n}(z))\,b(z,y) \\
    &= a\big(\varphi^{-n}(x), \gamma_a \varphi^{-n}(x)\big)\,b\big(\varphi^n \gamma_a \varphi^{-n}(x), \gamma_b \varphi^n \gamma_a \varphi^{-n}(x)\big)\\
    \\
    (b\alpha^n(a))(x,y) &= \sum_{(x,z) \in G^\lc(X, \varphi)} b(x,z)\,a(\varphi^{-n}(z),\varphi^{-n}(y)) \\
    &= b\big(x, \gamma_b(x)\big)\,a\big( \varphi^{-n}\gamma_b(x), \gamma_a \varphi^{-n} \gamma_b(x) \big) \,.
\end{align*}
Note that $(\alpha^n(a)b)(x,y)$ will be zero unless $x \in \varphi^n(K_a)$ and $\gamma_a ( \varphi^{-n} (x) ) \in \varphi^{-n}(K_b)$.  Now assume $n \geq N_0$, then we have \[ d(x, \varphi^n(\gamma_a(\varphi^{-n}(x)))) = d(\varphi^n(\varphi^{-n}(x)), \varphi^n(\gamma_a(\varphi^{-n}(x)))) < \varepsilon_0 \,. \]  By our assumptions this implies $x \in K_b' \subseteq U_b$.  Hence $\gamma_b(x)$ is defined, and we also have \[ d(\varphi^{-n}(x), \varphi^{-n}(\gamma_b(x))) < \varepsilon_0 \] showing also $\gamma_b(x) \in \varphi^n(K_a') \subseteq \varphi^{n}(U_a)$.  Likewise one may show that if $(b\alpha^n(a))(x,y)$ is non-zero, then $x \in \varphi^n(K_a')$ and $\gamma_a(\varphi^{-n}(x)) \in \varphi^{-n}(K_b')$ for $n \geq N_0$.  Hence, without loss of generality, if $(\alpha^n(a)b - b\alpha^n(a))(x,y)$ is non-zero then we can assume the above is true for $n \geq N_0$.

For simplicity, let us define the following quantities, which are all well-defined by the above.
\begin{align*}
    x' &= \varphi^n \gamma_a \varphi^{-n} (x) \\
    x'' &= \gamma_b(x) \\
    y &= \gamma_b \varphi^n \gamma_a \varphi^{-n} (x) = \gamma_b (x') \\
    y' &= \varphi^n \gamma_a \varphi^{-n} \gamma_b (x) = \varphi^n \gamma_a \varphi^{-n} (x'')
\end{align*}
Then in total we can rewrite $(\alpha^n(a)b - b\alpha^n(a))(x,y)$ as \[ (\alpha^n(a)b - b\alpha^n(a))(x,y) = a\big(\varphi^{-n}(x), \varphi^{-n}(x')\big)\,b(x',y) - b(x, x'')\,a\big(\varphi^{-n}(x''), \varphi^{-n}(y')\big) \,. \]
We then obtain the following inequality by adding and subtracting a cross-term to the right-hand side above
\begin{multline*}
\big| (\alpha^n(a)b - b\alpha^n(a))(x,y) \big| \leq \Big| b(x',y) \Big(a\big(\varphi^{-n}(x), \varphi^{-n}(x')\big) - a\big(\varphi^{-n}(x''), \varphi^{-n}(y')\big) \Big) \Big| \\
+ \Big| a\big(\varphi^{-n}(x''), \varphi^{-n}(y')\big) \Big( b(x,x'') - b(x',y) \Big) \Big|
\end{multline*}
When restricted to $K_b'$ the map $\gamma_b$ is uniformly continuous.  Coupled with the fact that $x$ is stably equivalent to $x'$ this means for any $\delta$ we can find an $N \geq N_0$ such that $d(x'',y) = d(\gamma_b(x), \gamma_b(x')) < \delta$ for all $n \geq N$ and $x,x' \in K_b'$.  By the same argument, since $x$ and $x''$ are unstably equivalent, we can also make $d(\varphi^{-n}(x'),\varphi^{-n}(y')) = d(\gamma_a\phi^{-n}(x), \gamma_a\phi^{-n}(x'')) < \delta$ for $x,x'' \in \varphi^n(K_a')$ small enough and all $n$ large enough.  Together with the fact that $a$ and $b$ are continuous functions, we can then, for some fixed $\epsilon > 0$, find an $N$ such that
\begin{align*}
    \Big| a\big(\varphi^{-n}(x), \varphi^{-n}(x')\big) - a\big(\varphi^{-n}(x''), \varphi^{-n}(y')\big) \Big| &< \frac{\epsilon}{2\|b\|} \text{, and} \\
    \big| b(x,x'') - b(x',y) \big| &< \frac{\epsilon}{2\|a\|}
\end{align*}
for all $n \geq N$.  Moreover, these inequalities hold for all $x \in \varphi^n(K_a') \cap K_b'$ because by assumption of local conjugacy we can make both $d(x,x') = d(x,\varphi^n\gamma_a\varphi^{-n}(x))$ and $d(x,x'') = d(x,\gamma_b(x))$ uniformly small for all $n$ large enough.  This gives us 
\[ \big\| \alpha^n(a)b - b\alpha^n(a) \big\| < \|b\| \frac{\epsilon}{2\|b\|} + \|a\| \frac{\epsilon}{2\|a\|} = \epsilon \]
and so we can conclude $\displaystyle \lim_{n \to \infty} \big\| \alpha^n(a)b - b\alpha^n(a) \big\| = 0$.

\end{proof}

\subsection{The Synchronizing Heteroclinic Algebras}\label{sec:heteroclinic-algebras}

We construct two $C^\ast$ algebras called the \emph{synchronizing heteroclinic} algebras.  However, since we are specializing to the case of synchronizing systems, we make some modifications to Thomsen's construction in \cite{thomsen2010c}.  In particular Thomsen utilizes the set of all \emph{post-periodic} points, which is defined as the set \[ X^\text{u} = \bigcup_{p \in \Per(X,\varphi)} X^\text{u}(p) \,. \]  This involves using all the periodic points in the dynamical system, and additionally requires the assumption that periodic points are dense.

However in synchronizing systems we have a preponderance of periodic points which are synchronizing and thus behave like the periodic points in a Smale space.  Our strategy is to restrict Thomsen's construction to just these synchronizing periodic points.  In fact, the analogy that synchronizing periodic points behave like the periodic points in a Smale space holds to the extent that the synchronizing heteroclinic algebras that we construct for different choices of synchronizing periodic points are all Morita equivalent, which parallels the result of Putnam and Spielberg \cite{putnam99}.

\begin{definition}
Let $(X,\varphi)$ be a synchronizing system and let $P \subseteq X$ be a finite set of synchronizing periodic points.  Then we define the \'{e}tale groupoids
\begin{enumerate}[(i)]
    \item $G^\lcs(X, \varphi, P) = \{ (x, y) \in X^\text{u}(P) \times X^\text{u}(P) \mid x \sim_\lcs y \}$, and
    \item $G^\lcu(X, \varphi, P) = \{ (x, y) \in X^\text{s}(P) \times X^\text{s}(P) \mid x \sim_\lcu y \}$.
\end{enumerate}
\end{definition}
\noindent Some remarks are in order:
\begin{enumerate}
\item We are using the inductive limit topology on $X^\text{u}(P)$ and $X^\text{s}(P)$ as discussed in Section \ref{sec:global-stable-unstable-sets}.  These groupoids are topologized from the stable and unstable local conjugacies respectively in the same way as with the homoclinic algebra.  
\item The unit space of $G^\lcs(X, \varphi, P)$ is $X^\text{u}(P)$ and the unit space of $G^\lcu(X, \varphi, P)$ is $X^\text{s}(P)$. 
\item Finally, it is often useful to assume that the set $P$ is $\varphi$-invariant; this is in fact needed if one wants to consider the $C^*$-automorphisms associated to $\varphi$. However, $\varphi$-invariance is not need for the construction of the groupoids and their $C^*$-algebras. This flexibility will be useful when we want to consider $P$ equal to the set containing a single periodic point, see for example Theorem \ref{th:sync-morita-equivalence} and Lemma \ref{lem:fp-lcu-lcs-dense}.
\end{enumerate}


\begin{definition}
Let $(X,\varphi)$ be a synchronizing system and let $P \subseteq X$ be a finite set of synchronizing periodic points.  Then we define
\begin{enumerate}[(i)]
    \item the \emph{stable synchronizing heteroclinic algebra}, $S(X,\varphi,P) = C^\ast_r(G^\lcs(X, \varphi, P))$, and
    \item the \emph{unstable synchronizing heteroclinic algebra}, $U(X,\varphi,P) = C^\ast_r(G^\lcu(X, \varphi, P))$
\end{enumerate}
which are both reduced groupoid $C^\ast$-algebras in the sense of Renault \cite{renault1980groupoid}.
\end{definition}

In the case of $(X,\varphi)$ being a Smale space, the stable and unstable synchronizing heteroclinic algebras are exactly the stable and unstable algebras defined by Putnam \cite{putnam99}.

\begin{theorem}\label{theorem:stable-unstable-independent-of-choice-of-P}
Let $(X,\varphi)$ be a synchronizing system and let $P,P' \in X$ be two finite sets of synchronizing periodic points.  Then
\begin{enumerate}[(i)]
    \item $S(X,\varphi,P)$ is Morita equivalent to $S(X,\varphi,P')$, and
    \item $U(X,\varphi,P)$ is Morita equivalent to $U(X,\varphi,P')$.
\end{enumerate}
\end{theorem}
\begin{proof}
We will only show the proof in the stable case.  We will prove that the groupoids $G^\lcs(X,\varphi,P)$ and $G^\lcs(X,\varphi,P \sqcup P')$ are Morita equivalent in the sense of \cite{farsi2018ample}.  By \cite[Proposition 3.10]{farsi2018ample} and \cite{sims12}, this Morita equivalence holds for reduced groupoid $C^\ast$-algebras.  In our situation, we have a functor \[ \iota : G^\lcs(X,\varphi,P) \to G^\lcs(X,\varphi,P \sqcup P') \] which is a continuous, open, inclusion of \'{e}tale groupoids.  By \cite[Definition 3.4]{farsi2018ample}, the functor $\iota$ is a weak equivalence if
\begin{enumerate}[(i)]
    \item the map $\{\gamma \in G^\lcs(X,\varphi,P \sqcup P') \mid r(\gamma) \in X^\text{u}(P) \} \to X^\text{u}(P \sqcup P')$ defined by \[ \gamma \mapsto s(\gamma) \] is an \'{e}tale surjection (\ie a surjective local homeomorphism, see \cite{moerdijk91}), and
    \item $G^\lcs(X,\varphi,P)$ is isomorphic to \[  \{\gamma \in G^\lcs(X,\varphi,P \sqcup P') \mid r(\gamma), s(\gamma) \in X^\text{u}(P) \} \,. \]
\end{enumerate}
Note that condition (ii) is true since the given set is simply the restriction of $G^\lcs(X,\varphi,P \sqcup P')$ to $G^\lcs(X,\varphi,P)$.  In order to show (i), we must show that for every $y \in X^\text{u}(P \sqcup P')$, there is an $x \in X^\text{u}(P)$ such that $x \sim_\lcs y$.  If $y \in X^\text{u}(P) \subseteq X^\text{u}(P \sqcup P')$, then we can just take $x = y$.  Otherwise, assume $y \in X^\text{u}(P')$ and since $y$ is synchronizing we let $R$ be a product neighborhood of $y$.  Since $X^\text{u}(P)$ is dense in $X$ (by Lemma \ref{lem:sync-per-stable-unstable-dense}), there exists $z \in R$ such that $z \in X^\text{u}(P)$.  Let $x = [y,z]$ so that $x \in X^\text{u}(P)$.  Then, by Lemma \ref{lem:rect-nbhd-implies-lcu-lcs}, we have $x \sim_\lcs y$.  Hence for all $y \in X^\text{u}(P \sqcup P')$ there is a $\gamma \in G^\lcs(X,\varphi,P\sqcup P')$ satisfying $s(\gamma) = y$.  

Furthermore, if $x \sim_\lcs y$ then for some $\delta,\delta' > 0$ there is a stable local conjugacy $\hat{\gamma} : X^\text{u}(x, \delta) \to X^\text{u}(y, \delta')$ which a homeomorphism onto its image.  Then $\{(z, \hat{\gamma}(z)) \mid z \in X^\text{u}(x, \delta)\}$ is a neighborhood of $\gamma$ in $G^\lcs(X,\varphi,P\sqcup P')$.  The map $\{(z, \hat{\gamma}(z)) \mid z \in X^\text{u}(x, \delta)\} \mapsto s(z, \hat{\gamma}(z)) = z$ is a homeomorphism since its inverse, $z \mapsto (z, \hat{\gamma}(z))$, is continuous (see \cite[Theorem 1.7]{thomsen2010c}).

Hence we have shown, without loss of generality, that $S(X,\varphi,P)$ is Morita equivalent to $S(X,\varphi,P \sqcup P')$ and $S(X,\varphi,P')$ is Morita equivalent to $S(X,\varphi,P \sqcup P')$.  Hence by transitivity we obtain a Morita equivalence between $S(X,\varphi,P)$ and $S(X,\varphi,P')$.
\end{proof}

\begin{remark} \label{rem:differentalgebras}
One could also define $C^*$-algebras using the set of all the synchronizing periodic points. These algebras would also be Morita equivalent to $S(X,\varphi,P)$ and $U(X, \varphi, P)$, respectively. The details are very similar to \cite[Lemma 4.16]{thomsen2010c}. On the other hand, Thomsen's construction of the heteroclinic algebras uses all the periodic points of the system; these algebras are not in general Morita equivalent to $S(X,\varphi,P)$ and $U(X, \varphi, P)$. 

Our choice to use synchronizing periodic points is based on the Morita equivalence between $S(X,\varphi,P) \otimes U(X, \varphi, P)$ and $\I_\sync(X,\varphi)$ that will be proved in the next section. Also, the choice to use only finitely many periodic orbits is based on the proofs in the finitely presented case, see Section \ref{denseFinitelyPresentedSubSec}.
\end{remark}

\subsection{The Synchronizing Ideal}

\begin{definition}\label{def:g-invariant}
Suppose $G$ is a locally compact Hausdorff \'{e}tale groupoid.  Then $X \subseteq G^0$ is called \emph{$G$-invariant} if for any $\gamma \in G$, $r(\gamma) \in X$ if and only if $s(\gamma) \in X$, 
\end{definition}

\begin{prop}\label{theorem:sync-ideal}
Let $(X,\varphi)$ be a synchronizing system with $A(X,\varphi)$ its homoclinic algebra, then there is an ideal $\mathcal{I}_\text{sync}(X,\varphi) \subseteq A(X,\varphi)$ determined by the synchronizing points in $X$.
\end{prop}
\begin{proof}
By Lemma \ref{lemma:sync-invariant-under-map} and Proposition \ref{theorem:sync-points-dense}, $X_\sync \subseteq X \cong \left(G^\lc(X, \varphi)\right)^0$ is an open $G^\lc(X,\varphi)$-invariant set.  Define $G^\lc_\sync(X, \varphi) = r^{-1}(X_\sync)$, then by \cite[Theorem 3.4.8]{putnamNotes19}, \[ \I_\sync(X,\varphi) = C^\ast_r(G^\lc_\sync(X, \varphi)) \] is an ideal of $A(X,\varphi)$.
\end{proof}

Recall the assumptions of Lemma \ref{lem:lcs-lcu-to-sync-existence} are that, for a mixing synchronizing system $(X,\varphi)$ and synchronizing periodic points $p,q \in X$,
\begin{itemize}
    \item $X^\lcs(x) \cap X^\text{u}(p)$ is dense in $X^\text{u}(p)$, and
    \item $X^\lcu(y) \cap X^\text{s}(q)$ is dense in $X^\text{s}(q)$
\end{itemize}
for all $x \in X^\text{u}(p)$ and $y \in X^\text{s}(q)$.  In Lemma \ref{lem:fp-lcu-lcs-dense} we will show that the above assumptions hold for finitely presented systems, in which case Theorem \ref{th:sync-morita-equivalence} holds.

\begin{theorem}\label{th:sync-morita-equivalence}
Let $(X,\varphi)$ be a mixing synchronizing system and $p$ a synchronizing periodic point.  In addition, assume that $X^\lcs(x) \cap X^\text{u}(p)$ is dense in $X^\text{u}(p)$ and $X^\lcu(y) \cap X^\text{s}(p)$ is dense in $X^\text{s}(p)$ for all $x \in X^\text{u}(p)$ and $y \in X^\text{s}(p)$.  Then $\I_\sync(X,\varphi)$ is Morita equivalent to $S(X,\varphi, p) \otimes U(X, \varphi, p)$.
\end{theorem}
\begin{proof}
For convenience we abbreviate the relevant groupoids as $G^\lc$, $G^\lcs$, and $G^\lcu$.  Furthermore we let $G^\lc_\sync$ be as in the proof of Proposition \ref{theorem:sync-ideal}.

Let $Z$ be the topological space \[ Z = \{ (x,y,z) \in X^\text{u}(p) \times X^\text{s}(p) \times X_\text{sync} \mid x \sim_\lcs z \sim_\lcu y \} \] with the topology described as follows.  Suppose $(x,y,z) \in Z$ and consider a product neighborhood $R \cong X^\text{u}(z, \delta_z) \times X^\text{s}(z, \delta_z)$ of $z$.  Furthermore, suppose $\delta > 0$ is small enough that $\gamma^\text{s} : X^\text{u}(x, \delta) \to X^\text{u}(z, \delta_z)$ and $\gamma^\text{u} : X^\text{s}(y, \delta) \to X^\text{s}(z, \delta_z)$ are, respectively, a stable and unstable local conjugacy giving the equivalences $x \sim_\lcs z \sim_\lcu y$.  Hence $[\gamma^\text{s}(x'), \gamma^\text{u}(y')]$ is well-defined for any $x' \in X^\text{u}(x, \delta)$ and $y' \in X^\text{s}(y, \delta)$.  For any $(x,y,z)$, $R$, and $\delta$ as above, we let the set \[ \left\{ \left(x', y', [\gamma^\text{s}(x'), \gamma^\text{u}(y')]\right) \mid x' \in X^\text{u}(x, \delta), y' \in X^\text{s}(y, \delta) \right\} \] be an open neighborhood of $(x,y,z)$ in $Z$.  These sets form a basis for a topology on $Z$.

We will show that $Z$ is a ($G^\lcs \times G^\lcu, G^\lc_\sync)$-equivalence in the sense of Definition 2.1 of \cite{mrw87}.  By Theorem 2.8 of \cite{mrw87}, $C_\text{c}(Z)$ can be completed into a $(S(X,\varphi,p) \otimes U(X,\varphi,p))-\I_\sync(X,\varphi)$ imprimitivity bimodule, and in particular this shows the desired Morita equivalence.  Note that this Morita equivalence still holds for reduced groupoid $C^\ast$-algebras \cite{sims12}.

Following \cite{mrw87}, we must show the following conditions in order for $Z$ to be a ($G^\lcs \times G^\lcu, G^\lc_\sync)$-equivalence:
\begin{enumerate}[(i)]
    \item $Z$ is a left principal $(G^\lcs \times G^\lcu)$-space,
    \item $Z$ is a right principal $G^\lc_\sync$-space,
    \item the $G^\lcs \times G^\lcu$ and $G^\lc_\sync$ actions commute,
    \item the map $\rho$ induces a bijection of $Z/G^\lc_\sync$ onto $(G^\lcs \times G^\lcu)^0 \cong X^\text{u}(p) \times X^\text{s}(p)$, and
    \item the map $\sigma$ induces a bijection of $(G^\lcs \times G^\lcu) \backslash Z$ of $(G^\lc_\sync)^0 \cong X_\sync$.
\end{enumerate}
We refer the reader to Section 2 of \cite{mrw87} for the relevant definitions.  In the literature the maps $\rho$ and $\sigma$ are called the \emph{moment maps}.  Note that condition (iii), that the actions of $G^\lcs \times G^\lcu$ and $G^\lc_\sync$ commute, will be clear from the definition of $Z$.

First we will show that $Z$ is a left principal $(G^\lcs \times G^\lcu)$-space.  Define the set \[ (G^\lcs \times G^\lcu) \ast Z = \left\{ \big( (x_2, x_1), (y_2, y_1), (x, y, z) \big) \mid x = x_1, y = y_1 \right\} \subseteq G^\lcs \times G^\lcu \times Z \,. \]
Then the action $\Gamma_l : (G^\lcs \times G^\lcu) \ast Z \to Z \times Z$ is defined as \[ \Gamma_l\big( (x_2, x_1), (y_2, y_1), (x_1, y_1, z) \big) = \big( (x_2, y_2, z), (x_1, y_1, z) \big) \,. \]  That this action is free is clear from the fact that $G^\lcs \times G^\lcu$ is principal.  To show it is proper we show that whenever $\{t_n\}_{n \geq 0}$ is a sequence in $(G^\lcs \times G^\lcu) \ast Z$ such that $\{\Gamma_l(t_n)\}_{n \geq 0}$ converges, then $\{t_n\}_{n \geq 0}$ has a convergent subsequence.  In fact we will show that $\{t_n\}_{n \geq 0}$ also converges whenever $\{\Gamma_l(t_n)\}_{n \geq 0}$ converges.  Assume that $\big( (x^n_2, x^n_1), (y^n_2, y^n_1), (x^n_1, y^n_1, z^n) \big)$ is a sequence in $(G^\lcs \times G^\lcu) \ast Z$ such that $\big( (x^n_2, y^n_2, z^n), (x^n_1, y^n_1, z^n) \big)$ converges to an element $\big( (x_2, y_2, z), (x_1, y_1, z) \big) \in Z \times Z$.  Observe that $x_1 \sim_\lcs x_2$ and $y_1 \sim_\lcu y_2$.  By our definition of the topology on $Z$, we have the convergence of $(x^n_2, x^n_1)$ to $(x_2,x_1)$ in $G^\lcs$ and $(y^n_2, y^n_1)$ to $(y_2, y_1)$ in $G^\lcu$.  Hence $\big( (x_2, x_1), (y_2, y_1), (x_1, y_1, z) \big)$ is in $(G^\lcs \times G^\lcu) \ast Z$ and it is the limit of the sequence $\big( (x^n_2, x^n_1), (y^n_2, y^n_1), (x^n_1, y^n_1, z^n) \big)$.

Next we must show that $Z$ is a right principal $G^\lc_\sync$-space.  Our method is similar to the above proof of (i); we define \[ Z \ast G^\lc_\sync = \left\{ \big( (x, y, z), (z_1, z_2) \big) \mid z = z_1 \right\} \subseteq Z \times G^\lc_\sync \] and an action $\Gamma_r: Z \ast G^\lc_\sync \to Z \times Z$ defined as \[ \Gamma_r\big( (x, y, z_1), (z_1, z_2) \big) = \big( (x,y,z_1), (x,y,z_2) \big) \] which is also free by the fact that $G^\lc_\sync$ is principal.  Assume that $\big( (x^n, y^n, z^n_1), (z^n_1, z^n_2) \big)$ is a sequence in $Z \ast G^\lc_\sync$ such that 
\[ \big( (x^n, y^n, z^n_1), (x^n, y^n, z^n_2) \big) \hbox{ converges to }\big( (x,y,z_1), (x,y,z_2) \big) \in Z \times Z. \]  Observe that $z_1 \sim_\lcs x \sim_\lcs z_2$ and $z_1 \sim_\lcu y \sim_\lcu z_2$, so $z_1 \sim_\lcs z_2$ and $z_1 \sim_\lcu z_2$.  Since $z_1$ and $z_2$ are also both synchronizing points in $X$ we can conclude from Proposition \ref{prop:sync-lcu-lcs-imply-lc} that $z_1 \sim_\lc z_2$.  Thus $\big( (x^n, y^n, z^n_1), (z^n_1, z^n_2) \big)$ converges to $\big( (x, y, z_1), (z_1, z_2) \big)$ in $Z \ast G^\lc_\sync$.  This shows $\Gamma_r$ is proper and that $Z$ is a right principal $G^\lc_\sync$-space.

To show (iv), we must show that $Z/G^\lc_\sync$, which are the orbits of $Z$ under the action of $G^\lc_\sync$, is in bijection with elements of $X^\text{s}(p) \times X^\text{u}(p)$ via the map induced by $\rho$ defined by $[(x,y,z)] \mapsto (x, y)$.  Injectivity of this map is clear from action of $G^\lc_\sync$.  To show it is also surjective we need to, given $(x,y) \in X^\text{u}(p) \times X^\text{s}(p)$, find $z \in X_\sync$ such that $x \sim_\lcs z \sim_\lcu y$.  This follows from Lemma \ref{lem:lcs-lcu-to-sync-existence}.

Lastly, to show (v), we note the injectivity of the map induced by $\sigma$, defined by $[(x,y,z)] \mapsto z$, is clear.  To show surjectivity we must similarly show that for every $z \in X_\sync$ there is $x \in X^\text{u}(p)$ and $y \in X^\text{s}(p)$ such that $x \sim_\lcs z \sim_\lcu y$.  This follows from Lemma \ref{lem:sync-to-lcs-lcu-existence}.
\end{proof}

\section{The Homoclinic Algebra of a Finitely Presented System} \label{SecFPS}

In this section we summarize the main results in the case of finitely presented systems.  First we must show that mixing finitely systems satisfy the hypothesis of Theorem \ref{th:sync-morita-equivalence}. Then, we discuss amenability of the stable and unstable synchronizing heteroclinic groupoids.

\subsection{Density of Stable/Unstable Local Conjugacy Equivalence Classes} \label{denseFinitelyPresentedSubSec}

We will show that mixing finitely presented systems satisfy the assumptions of Lemma \ref{lem:lcs-lcu-to-sync-existence}.  This result builds on work by Fisher on resolving extensions of finitely presented systems \cite{fisher13}.  First we must define some new terminology.

\begin{definition}
Let $(X,\varphi)$ and $(Y,\psi)$ be expansive dynamical systems and $\pi : Y \to X$ a factor map.  Then $\pi$ is called \begin{enumerate}[(i)]
    \item \emph{finite-to-one} if there is a constant $M > 0$ such that $|\pi^{-1}(x)| \leq M$ for all $x \in X$,
    \item \emph{u-resolving} (respectively \emph{s-resolving}) is $\pi$ restricted to $X^\text{u}(x)$ (respectively $X^\text{s}(x)$) is injective for all $x \in X$, and
    \item \emph{one-to-one almost everywhere} if there is a residual set of points in $X$ with unique pre-image under $\pi$.
\end{enumerate}
\end{definition}

We have the following result from \cite{fisher13}.   

\begin{theorem}{\cite[Theorem 1.1, Lemma 3.2]{fisher13}}\label{theorem:fp-factor-of-smale}
Suppose $(X,\varphi)$ is an irreducible finitely presented system.  Then there is an irreducible Smale space $(Y,\psi)$ and a u-resolving factor map $\pi:Y \to X$ such that $\pi$ is one-to-one almost everywhere.  Moreover, there exists a dense open set $W \subseteq X$ such that each periodic point in $W$ has a unique pre-image under $\pi$.
\end{theorem}

Likewise, there is a parallel to Theorem \ref{theorem:fp-factor-of-smale} for the existence of an s-resolving factor map with identical properties.  One can also show that the Smale space $(Y,\psi)$ in Theorem \ref{theorem:fp-factor-of-smale} can additionally be shown to be mixing whenever $(X,\varphi)$ is mixing, see Lemma \ref{lemma:fp-factor-mixing}.

We now show that if $\pi : Y \to X$ is a u-resolving factor map as in Theorem \ref{theorem:fp-factor-of-smale}, it must be finite-to-one.  The technique of the proof is essentially the same as \cite[Theorem 2.5.3]{MR3243636}.  Note that $\pi$ being a finite-to-one factor map is a necessary assumption in \cite[Lemma 3.2]{fisher13}.

\begin{lemma} \label{lemma:FiniteToOne}
Let $(X,\varphi)$ be a finitely presented system, $(Y,\psi)$ a Smale space, and $\pi : Y \to X$ is a factor map which is u-resolving or s-resolving.  Then there exists $M \in \N$ such that the following hold:
\begin{enumerate}
\item For any $x\in X$, there exists $y_1, \ldots , y_K$ in $Y$ with $K\le M$ such that 
\[ \pi^{-1}(X^u(x))= \cup_{k=1}^{K} Y^u(y_k). \]
\item For any $x\in X$, there exists $y_1, \ldots , y_L$ in $Y$ with $L\le M$ such that 
\[ \pi^{-1}(X^s(x))= \cup_{l=1}^{L} Y^s(y_l). \]
\item For any $x\in X$, $|\pi^{-1}(x)| \le M$, so that in particular, $\pi$ is finite-to-one.
\end{enumerate}
\end{lemma}
\begin{proof}
We will prove the first and last statements when $\pi$ is u-resolving. The other cases are similar; the interested reader can see the proof of \cite[Theorem 2.5.3]{MR3243636} for more details.  

Suppose $y \sim_\text{u} y'$ for $y, y' \in Y$ and denote $x = \pi(y)$ and $x' = \pi(y')$.  Then since $\pi$ is a factor map \[ \lim_{n \to \infty} d(\varphi^{-n}(x), \varphi^{-n}(x')) = \lim_{n \to \infty} d(\pi(\psi^{-n}(y)), \pi(\psi^{-n}(y'))) = 0 \] by the uniform continuity of $\pi$.  Hence $Y^\text{u}(y) \subseteq \pi^{-1}(X^\text{u}(x))$ for any $x \in X$ and $y \in \pi^{-1}(x)$.

We will now show item (1) above, by finding the required $M$.  By \cite[Lemma 3.3]{fisher13}, there is a constant $\delta_\pi > 0$ such that if $y_1, y_2 \in Y$ and $d(y_1, y_2) < \delta_\pi$ then $\pi([y_1,y_2]) = [\pi(y_1), \pi(y_2)]$.  Cover $Y$ with balls of radius $\displaystyle \frac{\delta_\pi}{2}$, then take a finite subcover $\displaystyle \left\{ B_m \right\}_{m = 1}^M$.  Suppose $y_1, y_2, \dots, y_{M+1} \in Y$ are such that $\pi(y_i) \sim_\text{u} \pi(y_j)$ for all $1 \leq i,j \leq M + 1$.  Fix $\displaystyle 0 < \varepsilon \leq \frac{\varepsilon_X}{2}$, then choose $n \geq 0$ such that $\varphi^{-n}(\pi(y_i)) \in X^\text{u}(\varphi^{-n}(\pi(y_j)), \varepsilon)$ for all $1 \leq i,j \leq M + 1$.  By the pigeonhole principle, there exists $i \neq j$ such that $\psi^{-n}(y_i), \psi^{-n}(y_j) \in B_m$ for some $1 \leq m \leq M$.  Hence $d(\psi^{-n}(y_i), \psi^{-n}(y_j)) < \delta_\pi$, so
\begin{align*}
\pi([\psi^{-n}(y_i), \psi^{-n}(y_j)]) &= [\pi(\psi^{-n}(y_i)), \pi(\psi^{-n}(y_j))] \\
&= [\varphi^{-n}(\pi(y_i)), \varphi^{-n}(\pi(y_j))] \,.
\end{align*}
Then by our assumption that $\varphi^{-n}(\pi(y_i)) \in X^\text{u}(\varphi^{-n}(\pi(y_j)), \varepsilon)$, it follows from expansiveness that \[ [\varphi^{-n}(\pi(y_i)), \varphi^{-n}(\pi(y_j))] = \varphi^{-n}(\pi(y_j)) = \pi(\psi^{-n}(y_j)) \,. \]
Since $\pi$ is u-resolving, we have $[\psi^{-n}(y_i), \psi^{-n}(y_j)] = \psi^{-n}(y_j)$.  It follows that $\psi^{-n}(y_i) \sim_\text{u} \psi^{-n}(y_j)$ and consequently $y_i \sim_\text{u} y_j$.

Suppose $x \in X$ and $\pi^{-1}(x) = \{ y_1, y_2, \dots, y_{M+1} \}$.  Then we have shown there exists $i \neq j$ such that $y_i \sim_\text{u} y_j$.  However since $\pi$ is u-resolving and $\pi(y_i) = \pi(y_j)$, we have $y_i = y_j$.  Hence $|\pi^{-1}(x)| \leq M$ for all $x \in X$.
\end{proof}

\begin{lemma}\label{lemma:fp-factor-sync-per-unique-pre-image}
Suppose $(X,\varphi)$ is an irreducible finitely presented system and $\pi:Y \to X$ is a factor map as in Theorem \ref{theorem:fp-factor-of-smale}.  Then there exists a synchronizing periodic point $p$ such that $p$ has a unique pre-image under $\pi$.
\end{lemma}
\begin{proof}
Note that synchronizing periodic points are dense in $X$.  This is because $\Per(X,\varphi)$ is dense in $X$ by Corollary \ref{cor:sync-periodic-dense}, and $X_\sync$ is an open dense set by Proposition \ref{theorem:sync-points-dense}.  Then it follows from Theorem \ref{theorem:fp-factor-of-smale} that there is a synchronizing periodic point $p \in W$, in other words $p$ is a synchronizing periodic point with unique pre-image under $\pi$.
\end{proof}

\begin{lemma}\label{lemma:fp-factor-pi-restriction-homeo}
Suppose $(X,\varphi)$ and $(Y,\psi)$ are expansive dynamical systems and $\pi : Y \to X$ is a factor map.  If $p \in X$ is a periodic point with a unique pre-image $q \in Y$, then
\begin{enumerate}[(i)]
    \item $\pi^{-1}(X^\text{s}(p)) = Y^\text{s}(q)$,
    \item $\pi^{-1}(X^\text{u}(p)) = Y^\text{u}(q)$,
    \item if $\pi$ is s-resolving, then $\pi$ restricted to $Y^\text{s}(q)$ is a homeomorphism from $Y^\text{s}(q)$ onto $X^\text{s}(p)$, and
    \item if $\pi$ is u-resolving, then $\pi$ restricted to $Y^\text{u}(q)$ is a homeomorphism from $Y^\text{u}(q)$ onto $X^\text{u}(p)$.
\end{enumerate}
\end{lemma}
\begin{proof}
We will first show (ii), then the proof of (i) is similar.  Let $r$ denote the period of $p$.  Suppose $y \in Y$ is contained in $\pi^{-1}(x)$.  By compactness, the sequence $\{ \psi^{-rn}(y) \}_{n \geq 0}$ has a subsequence $\{ \psi^{-rn_k}(y) \}_{k \geq 0}$ converging to a point $q' \in Y$.  Then for $x \in X^\text{u}(p)$ we have
\begin{align*}
    \pi(q') &= \lim_{k \to \infty} \pi\parens{\psi^{-rn_k}(y)} \\
    &= \lim_{k \to \infty} \varphi^{-rn_k}(\pi(y)) \\
    &= \lim_{k \to \infty} \varphi^{-rn_k}(x) \\
    &= p \,.
\end{align*}
Since the pre-image of $p$ is unique we must have $q' = q$.  Furthermore, since $\{ \psi^{-rn}(y) \}_{n \geq 0}$ has a convergent subsequence and since every subsequence must converge to $q$, we have $\displaystyle \lim_{n \to \infty} \psi^{-rn}(y) = q$.  Fix $\varepsilon > 0$ and let $\delta > 0$ be such that $d(q, w) < \delta$ implies $d(\psi^{-i}(q), \psi^{-i}(w)) < \varepsilon$ for $0 \leq i < r$.  Then there is an $N$ such that $d(q, \psi^{-rn}(y)) < \delta$ for all $n \geq N$.  It follows that $d(\psi^{-j}(q), \psi^{-j}(y)) < \varepsilon$ for all $j \geq rN$.  Hence $y \in Y^\text{u}(q)$, and since $\pi(Y^\text{u}(q)) \subseteq X^\text{u}(p)$, this proves (ii).

Next we will show (iv), and proof of (iii) is similar.  If $\pi$ is u-resolving, then in particular we have shown $\pi|_{Y^\text{u}(q)}$ is a bijection.  Next we show that $\pi|_{Y^\text{u}(q)}$ is continuous.  Let $\{y_n\}_{n \geq 0}$ be a sequence converging to $y \in Y^\text{u}(q)$.  Fix $\varepsilon > 0$ and let $\delta > 0$ be such that $d(w,w') < \delta$ implies $d(\pi(w), \pi(w')) < \varepsilon$ for all $w,w' \in Y$.  Then, there exists $N$ such that $y_n \in Y^\text{u}(y, \delta)$ for all $n \geq N$.  In other words, $d(\psi^{-k}(y_n), \psi^{-k}(y)) \leq \delta$ for all $k \geq 0$ and $n \geq N$.  Let $x_n = \pi(y_n)$ for all $n \geq 0$ and also $x = \pi(y)$.  Consequently,
\begin{align*}
    d(\varphi^{-k}(x_n), \varphi^{-k}(x)) &= d(\varphi^{-k}(\pi(y_n)), \varphi^{-k}(\pi(y))) \\
    &= d(\pi(\psi^{-k}(y_n)), \pi(\psi^{-k}(y))) \\
    &< \varepsilon
\end{align*}
for all $k \geq 0$ and $n \geq N$.  Hence $x_n \in X^\text{u}(x, \varepsilon)$ for all $n \geq N$ and this shows $\pi|_{Y^\text{u}(q)}$ is indeed continuous.

Lastly we will show $\pi|_{Y^\text{u}(q)}$ is proper.  Assume $\{y_n\}_{n \geq 0}$ is a sequence in $Y^\text{u}(q)$ such that the sequence $\{x_n\}_{n \geq 0}$, with $x_n = \pi(y_n)$, converges to $x \in X^\text{u}(p)$.  Let $y$ be the unique pre-image of $x$ under $\pi$.  Then by the compactness of $Y$, $\{y_n\}_{n \geq 0}$ has a subsequence $\{y_{n_k}\}_{k \geq 0}$ converging to some $y' \in Y$.  By continuity $\pi(y') = x$, but since $y$ is the unique pre-image of $x$ we must have $y' = y$.  Hence it is actually the case that $y_n$ converges to $y$, and this shows the properness of $\pi|_{Y^\text{u}(q)}$.  Since $Y^\text{u}(q)$ and $X^\text{u}(p)$ are locally compact Hausdorff spaces, properness then implies $\pi|_{Y^\text{u}(q)}$ is a homeomorphism.
\end{proof}

We now show that the Smale space $(Y,\psi)$ in Theorem \ref{theorem:fp-factor-of-smale} can additionally be chosen to be mixing if $(X,\varphi)$ is mixing.  The technique of the proof is essentially the same as \cite[Proposition 4.8]{amini17}.

\begin{lemma}\label{lemma:fp-factor-mixing}
Suppose $(X,\phi)$ is a mixing finitely presented system, $(Y,\psi)$ an irreducible Smale space, and $\pi: Y \to X$ a factor map which is one-to-one almost everywhere.  Then $(Y,\psi)$ is mixing also.
\end{lemma}
\begin{proof}
The set $\Per(Y,\psi)$ is dense in $Y$ by Proposition \ref{theorem:sync-points-dense}. Hence $\pi\parens{\Per(Y,\psi)} \subseteq \Per(X, \varphi)$ is dense in $X$ because $\pi$ is continuous and onto. Let $W$ be as in Theorem \ref{theorem:fp-factor-of-smale}, then since $W \cap X_\sync$ is an open dense set, there exists $p \in \pi\parens{\Per(Y,\psi)} \cap \parens{X_\sync \cap W}$.  Hence, $p$ is periodic, synchronizing, and has unique pre-image $q \in Y$ such that $q$ is also periodic.  Furthermore, observe that $\varphi(p)$ is also periodic, synchronizing, and has the unique pre-image $\psi(q)$.

By Lemma \ref{lemma:fp-factor-pi-restriction-homeo}, since $p$ and $\varphi(p)$ are both periodic with unique pre-image we have
\begin{align*}
    \pi^{-1}(X^\text{s}(p)) &= Y^\text{s}(p) \,, \\
    \pi^{-1}(X^\text{u}(p)) &= Y^\text{u}(p) \,, \\
    \pi^{-1}(X^\text{s}(\varphi(p))) &= Y^\text{s}(\psi(q)) \text{ , and } \\
    \pi^{-1}(X^\text{u}(\varphi(p))) &= Y^\text{u}(\psi(q)) \,.
\end{align*}
Since $(X,\varphi)$ is mixing and $p$ and $\varphi(p)$ are synchronizing periodic points, we know from Lemma \ref{lem:sync-per-stable-unstable-dense} that both $X^\text{s}(p) \cap X^\text{u}(\varphi(p))$ and $X^\text{s}(\varphi(p)) \cap X^\text{u}(p)$ are dense in $X$ and in particular non-empty.  Consequently,
\begin{align*}
    Y^\text{s}(q) \cap Y^\text{u}(\psi(q)) &= \pi^{-1}(X^\text{s}(p)) \cap \pi^{-1}(X^\text{u}(\varphi(p))) = \pi^{-1}(X^\text{s}(p) \cap X^\text{u}(\varphi(p))) \neq \emptyset \text{ , and } \\
    Y^\text{s}(\psi(q)) \cap Y^\text{u}(q) &= \pi^{-1}(X^\text{s}(\varphi(p))) \cap \pi^{-1}(X^\text{u}(p)) = \pi^{-1}(X^\text{s}(\varphi(p)) \cap X^\text{u}(p)) \neq \emptyset \,.
\end{align*}
Since $(Y,\psi)$ is an irreducible Smale space, we have from \cite[Lemma 4.7]{amini17} that the above is equivalent to $(Y,\psi)$ being mixing.
\end{proof}

Finally, we will show the desired density of stable and unstable local conjugacy equivalence classes for the case of a mixing finitely presented system.

\begin{lemma}\label{lem:fp-lcu-lcs-dense}
Let $(X, \varphi)$ be a mixing finitely presented system. There exists $p \in X$ a synchronizing periodic point such that for all $x \in X^\text{u}(p)$ and $y \in X^\text{s}(p)$,
\begin{enumerate}[(i)]
    \item $X^\lcs(x) \cap X^\text{u}(p)$ is dense in $X^\text{u}(p)$, and
    \item $X^\lcu(y) \cap X^\text{s}(p)$ is dense in $X^\text{s}(p)$.
\end{enumerate}
\end{lemma}
\begin{proof}
Let $\pi : Y \to X$ be a factor map as in Theorem \ref{theorem:fp-factor-of-smale} where $(Y,\psi)$ is a mixing Smale space by Lemma \ref{lemma:fp-factor-mixing}.  By Lemma \ref{lemma:fp-factor-sync-per-unique-pre-image} we can find a synchronizing periodic point $p \in X$ such that $p$ has a unique pre-image $q$ under $\pi$.  Moreover, by Lemma \ref{lemma:fp-factor-pi-restriction-homeo}, we have that the restriction $\pi|_{Y^\text{u}(q)} : Y^\text{u}(q) \to X^\text{u}(p)$ is a homeomorphism.  To simplify notation we will denote the restriction $\pi|_{Y^\text{u}(q)}$ as $\pi_q$.  Crucially, we will use the fact that $Y^\text{s}(y)\cap Y^\text{u}(q)$ is dense in $Y^\text{u}(q)$ since $(Y,\psi)$ is a mixing Smale space \cite{ruelle_2004}.

We will show the proof of (i). Consider $y_1, y_2 \in Y^\text{u}(q)$ such that $y_1 \sim_\text{s} y_2$, and denote $x_1 = \pi(y_1)$ and $x_2 = \pi(y_2)$.  We will construct a stable local conjugacy from $x_1$ to $x_2$.  First note that since $y_1 \sim_\text{s} y_2$ and $y_1 \sim_\text{u} q \sim_\text{u} y_2$, we have $y_1 \sim_\text{h} y_2$.  Because $(Y,\psi)$ is a Smale space we can apply Lemma \ref{lem:smale-h-implies-lc}, and the fact that $y_1 \sim_\lc y_2$ implies $y_1 \sim_\lcs y_2$, to obtain a stable local conjugacy $\gamma : Y^\text{u}(y_1, \delta_1) \to Y^\text{u}(y_2, \delta_2)$.

Next we will show that $\gamma$ descends to a stable local conjugacy from $x_1$ to $x_2$.  Let $\delta > 0$ be such that $\pi^{-1}(X^\text{u}(x_1, \delta)) \subseteq Y^\text{u}(y_1, \delta_1)$, then define $\widetilde{\gamma}(z)$ be defined as \[ \widetilde{\gamma}(z) = (\pi \circ \gamma \circ \pi_q^{-1})(z) \] for $z \in X^\text{u}(x_1, \delta)$.  Let $V$ denote the image of $\widetilde{\gamma}$ in $X^\text{u}(p)$, then $\widetilde{\gamma} : X^\text{u}(x_1, \delta) \to V$ is a homeomorphism such that $\widetilde{\gamma}(x_1) = x_2$.  Then we compute
\begin{align*}
    \lim_{n \to \infty} \sup_{z \in X^\text{u}(x_1, \delta)} d(\varphi^n(z), \varphi^n(\widetilde{\gamma}(z))) &= \lim_{n \to \infty} \sup_{z \in X^\text{u}(x_1, \delta)} d(\varphi^n(z), \varphi^n((\pi \circ \gamma \circ \pi_q^{-1})(z))) \\
    &= \lim_{n \to \infty} \sup_{w \in \pi^{-1}(X^\text{u}(x_1, \delta))} d(\varphi^n(\pi(w)), \varphi^n(\pi(\gamma(w)))) \\
    &= \lim_{n \to \infty} \sup_{w \in \pi^{-1}(X^\text{u}(x_1, \delta))} d(\pi(\psi^n(w)), \pi(\psi^n(\gamma(w)))) \\
    &= 0
\end{align*}
Where the last line is by the uniform continuity of $\pi$.  Note that the metric $d$ is the metric on $X$.  Hence $x_1 \sim_\lcs x_2$.

Let $x \in X^\text{u}(p)$ and let $y \in Y^\text{u}(q)$ denote the unique pre-image of $x$ under $\pi$.  Then since $Y^\text{s}(y)\cap Y^\text{u}(q)$ is dense in $Y^\text{u}(q)$, and because $\pi_q$ is a homeomorphism, we have \[ \pi_q\left(Y^\text{s}(y) \cap Y^\text{u}(q)\right) \subseteq X^\lcs(x) \cap X^\text{u}(p) \] is dense in $X^\text{u}(p)$.  This proves (i).
\end{proof}

\subsection{Amenability}
It is an open question whether the homoclinic groupoid of an expansive dynamical system is amenable. In this section, we prove that the groupoids $G^{{\rm lcs}}(X, \varphi, P)$ and $G^{{\rm lcu}}(X, \varphi, P)$ are amenable when $(X, \varphi)$ is mixing and finitely presented. 
\begin{lemma}
Suppose that $(X, \varphi)$ is a mixing finitely presented system, $(Y, \psi)$ is a mixing Smale space, $\pi: Y \rightarrow X$ is almost one-to-one u-resolving factor map, and $P$ is a finite set of synchronizing periodic points for which $\pi$ is one-to-one. Then the map $\pi \times \pi$ maps $G^s(Y, \psi, \pi^{-1}(P))$ to an open subgroupoid of $G^{lcs}(X, \varphi, P)$. 
\end{lemma}
\begin{proof}
It follows from Lemma \ref{lemma:fp-factor-pi-restriction-homeo} that $\pi \times \pi$ maps the unit space of $G^s(Y, \psi, \pi^{-1}(P))$ homeomorphically to the unit space of $G^{lcs}(X, \varphi, P)$. Furthermore, in the proof of Lemma \ref{lem:fp-lcu-lcs-dense} it was shown that $\pi \times \pi$ maps a basic set used to defined the topology on $G^s(Y, \psi, \pi^{-1}(P))$ to a basic set for the topology on $G^{lcs}(X, \varphi, P)$. Thus, $\pi \times \pi$ is an open map.
\end{proof}
We can now prove the main result of this section. The reader less familiar with amenability might find it useful to review \cite{MR3403785} (see in particular Remark 2.5 of \cite{MR3403785}).
\begin{theorem} \label{thm:amenable}
Suppose that $(X, \varphi)$ is a mixing finitely presented system and $P$ is a finite set of synchronizing periodic points. Then the groupoids $G^{{\rm lcs}}(X, \varphi, P)$ and $G^{{\rm lcu}}(X, \varphi, P)$ are amenable. 
\end{theorem}
\begin{proof}
To begin, we recall that $G^{{\rm lcs}}(X, \varphi, P)$ does not have the subspace topology as subset of $X^u(P)\times X^u(P)$ with the product topology. However, as Borel groupoids, we have that $G^{{\rm lcs}}(X, \varphi, P)$ with the usual topology is the same as $G^{{\rm lcs}}(X, \varphi, P)$ with the subspace topology. This follows from the fact that the basics sets used to defined the topology on $G^{{\rm lcs}}(X, \varphi, P)$ are Borel sets with respect to the subspace topology.

We will only consider  $G^{{\rm lcs}}(X, \varphi, P)$ in detail. By Theorem \ref{theorem:fp-factor-of-smale}, there exists $(Y, \psi)$ a mixing Smale space and $\pi: Y \rightarrow X$ almost one-to-one u-resolving factor map. Since different choices of synchronizing periodic points only affect the groupoids up to Morita equivalence, we can use Lemma \ref{lemma:fp-factor-sync-per-unique-pre-image} to ensure that we take $P$ that satisfies the conditions in the previous lemma (here we note that if $\pi$ is one-to-one for a point, then it is one-to-one for that point's entire orbit).

By \cite[Theorem 1.1]{putnam99}, $G^s(Y, \psi, \pi^{-1}(P))$ is amenable and hence by the main result of \cite{MR3403785}, $G^s(Y, \psi, \pi^{-1}(P))$ is Borel amenable. The previous lemma and item (2) of Lemma \ref{lemma:FiniteToOne} are the hypotheses of Proposition 2.9 Part (vii) in \cite{MR1900547}. It follows from this proposition that $G^{{\rm lcs}}(X, \varphi, P)$ is Borel amenable and again using the main result of \cite{MR3403785} it follows that $G^\lcs(X, \psi, P)$ is amenable. 
\end{proof}

\subsection{Summary of Results}

\begin{theorem}\label{th:fp-main-results}
Suppose $(X,\varphi)$ is a mixing finitely presented system and $P \subseteq X$ a finite set of synchronizing periodic points. Then, the groupoids $G^{lc}_{\sync}(X, \varphi)$, $G^\lcs(X,\varphi,P)$, and $G^\lcu(X,\varphi,P)$ are each amenable and \[ 0 \longrightarrow \I_\sync(X,\varphi) \longrightarrow A(X,\varphi) \longrightarrow A(X,\varphi)/\I_\sync(X, \varphi) \longrightarrow 0 \] is an exact sequence of $C^\ast$-algebras, where
\begin{enumerate}[(i)]
    \item $\I_\sync(X,\varphi)$ is Morita equivalent to $S(X,\varphi,P) \otimes U(X,\varphi,P)$ and
    \item $\I_\sync(X,\varphi)$, $S(X,\varphi,P)$, and $U(X,\varphi,P)$ are all simple.
\end{enumerate}
\end{theorem}
\begin{proof}
Combining Theorem \ref{th:sync-morita-equivalence} and Lemma \ref{lem:fp-lcu-lcs-dense}, we have that there exists $p\in X$ a synchronizing periodic point such that $\I_\sync(X,\varphi)$ is Morita equivalent to $S(X,\varphi, p) \otimes U(X, \varphi, p)$. The Morita equivalence in (i) then follows from Theorem \ref{theorem:stable-unstable-independent-of-choice-of-P}. This also implies that $G^{lc}_{\sync}(X, \varphi)$ is amenable since $G^\lcs(X,\varphi,P)$, and $G^\lcu(X,\varphi,P)$ are amenable by Theorem \ref{thm:amenable}. 

From Lemma \ref{lem:fp-lcu-lcs-dense} and \cite[Proposition 4.3.7]{sims2018etale}, it follows that $S(X,\varphi,P)$ and $U(X,\varphi,P)$ are both simple.  It then follows from (i) that $\I_\sync(X,\varphi)$ is also simple. 

\end{proof}


It is not presently known if $G^\lc(X,\varphi)$ is amenable even when $(X, \varphi)$ is a finitely presented system. However, there are no amenability issues when $(X, \varphi)$ is mixing, finitely presented, and the set $X \setminus X_\sync$ is finite. We have the following in this situation.

\begin{theorem}\label{th:fp-main-results-finite}
Suppose $(X,\varphi)$ is a mixing finitely presented system where $X \setminus X_\sync$ is finite and $P \subseteq X$ a finite set of synchronizing periodic points. Then the groupoids $G^\lc(X,\varphi)$, $G^{lc}_{\sync}(X, \varphi)$, $G^\lcs(X,\varphi,P)$, and $G^\lcu(X,\varphi,P)$ are each amenable. Moreover, 
 \[ 0 \longrightarrow \I_\sync(X,\varphi) \longrightarrow A(X,\varphi) \longrightarrow \mathbb{C}^{|X \setminus X_\sync|} \longrightarrow 0 \] 
 is an exact sequence of $C^\ast$-algebras, where (as in the previous theorem)
\begin{enumerate}[(i)]
    \item $\I_\sync(X,\varphi)$ is Morita equivalent to $S(X,\varphi,P) \otimes U(X,\varphi,P)$ and
    \item $\I_\sync(X,\varphi)$, $S(X,\varphi,P)$, and $U(X,\varphi,P)$ are all simple
\end{enumerate}
\end{theorem}
\begin{proof}
Based on the previous theorem, we need to show that $G^\lc(X, \varphi)$ is amenable and that $A(X,\varphi)/\I_\sync(X, \varphi) \cong \mathbb{C}^{|X \setminus X_\sync|}$. 

For the first of these statements, using \cite[Proposition 4.3.2]{sims2018etale}, the full groupoid $C^*$-algebra of $G^\lc(X, \varphi)$ fits within a short exact sequence with the ideal $\I_\sync(X,\varphi)$ (where we have used the fact the groupoid restricted to the synchronizing points is amenable). The quotient by this ideal is finite dimensional because there are finitely many non-synchronizing points. In particular, both $\I_\sync(X, \varphi)$ and this quotient algebras are nuclear so the full groupoid $C^*$-algebra of $G^\lc(X, \varphi)$ is also nuclear. It follows that $G^\lc(X, \varphi)$ is amenable, see \cite[Theorem 4.1.5]{sims2018etale}.

We now prove the second part. By \cite[Propsition 4.3.2]{sims12}, $A(X,\varphi)/\I_\sync(X, \varphi)$ is the reduced $C^\ast$-algebra of the local conjugacy groupoid restricted to $X \setminus X_\sync$.  Hence, $A(X,\varphi)/\I_\sync(X, \varphi)$ is a finite dimensional $C^\ast$-algebra. Moreover, since $X \setminus X_\sync$ is finite, every $q \in X \setminus X_\sync$ must be a periodic point so Lemma \ref{lem:periodic-h-implies-equals} implies that if $p \sim_\lc q$ with $p$, $q$ in $X \setminus X_\sync$, then $p = q$. It follows that \[ A(X,\varphi)/\I_\sync(X, \varphi) \cong \mathbb{C}^{|X \setminus X_\sync|} \,. \]
\end{proof}
\begin{example}
Let $(X_\ev,\sigma)$ denote the even shift. It is not hard to show that in the even shift there is only one non-synchronizing element, namely the sequence of all zeros, see \cite[Example 3.1.1]{AndrewPhDThesis} for details. Furthermore the even shift is a sofic shift, so in particular it is finitely presented. The previous theorem implies that all the relevant groupoids are amenable (this also follows from the fact that the relevant groupoids are AF). Hence, from Theorem \ref{th:fp-main-results-finite} we have the exact sequence \[ 0 \longrightarrow \I_\sync(X_\ev, \sigma) \longrightarrow A(X_\ev,\sigma) \longrightarrow \C \longrightarrow 0 \,. \]
\end{example}

\begin{example}
Recall the classification of expansive homeomorphisms of orientable surfaces discussed in Theorem \ref{thm:surfaces}.  Let $(M,\varphi)$ be an expansive homeomorphism of an orientable surface. 

 If $M$ is the 2-torus, them $(M,\varphi)$ is a Smale space.  Hence $M_\sync = M$ and $\I_\sync(M, \varphi) = A(M,\varphi)$.

If the genus of $M$ is larger than one, then by \cite[Lemma 4.9]{artigue2008local} (also see \cite{Lewowicz08}), $M \setminus M_\sync$ is finite, say $|M \setminus M_\sync| = n$. Then by Theorem \ref{th:fp-main-results-finite}, all the relevant groupoids are amenable and we have the exact sequence \[ 0 \longrightarrow \I_\sync(M,\varphi) \longrightarrow A(M,\varphi) \longrightarrow \C^n \longrightarrow 0 \,. \]
\end{example}


\end{document}